\documentclass{amsart}
\usepackage{amsthm}
\usepackage{amsmath}
\usepackage{amsfonts}
\usepackage{amssymb}
\usepackage[final,notref,notcite]{showkeys}
\def\real{{\mathbb{R}}}%real numbers
\def\meas{\mu} %measure
\def\metric{\rho} %symbol for the metric
\def\dist#1,#2.{\metric(#1,#2)} %distance between two points
\newcommand{\dermod}[1][U]{{\rm Der}(#1,\meas)} %local module of derivations
\newcommand{\dermodsp}{{\rm Der}(X,\meas)} %global module of derivations
\def\elleinfty#1.{{\rm L}^\infty(#1,\meas)} %L^\infty functions
\def\ellep#1.{{\rm L}^p(#1,\meas)} %L^p functions
\def\introellep{{\rm L}^p}%summy symbol for L^p functions
\def\presob#1.{{{\mathcal D}(#1,p)}}%presobolev space
\def\sections{\Gamma(T^*X)}%sections of the tangent bundle
\def\lipsob#1.{{\rm H}^{1,p}(#1,\meas)}%sobolev space obtained by closing Lipschitz functions
\def\introlipsob{{\rm H}^{1,p}}%dummy lipsob in the introduction
\def\lipalg#1.{{\rm Lip}^{\infty}(#1)} %Lipschitz algebra of bounded functions
\def\plipalg#1.{{\rm Lip}_0(#1,x_0)} %pointed Lipschitz algebra
\def\lipfun#1.{{\rm Lip}(#1)}%set of Lipschitz functions
\def\glip#1.{{\bf L}(#1)} %global Lipschitz constant
\def\vnorm#1!{{\|#1\|}} %general symbol for norm
\def\supnorm#1!{{\vnorm#1!}_\infty}
\def\lipnorm#1!{{\vnorm#1!}_{\lipalg X.}} %norm on the Lipschitz algebra
\def\plipnorm#1!{{\vnorm#1!}_{\plipalg X.}}%norm on the pointed Lipschitz algebra
\def\arenseels#1.{{\rm AE}[#1]}%Arens-Eels space
\def\realspan#1.{\real\langle#1\rangle}% real span
\def\genset{\left\{g_j\right\}_{j=1}^M} %generating set
\def\biglip{\pounds} %local lipschitz constant
\def\varlip{{\rm Var}} %variation over a ball 
\def\smllip{\ell} %minimal variation
\def\elleoneloc#1.{{\rm L}_{\rm loc}^1(#1,\meas)}%L^1_loc
\def\precise#1{{#1}^\bigstar}%defines the precise representative
\DeclareMathOperator\avint{\int\!\!\!\!\!--}
\def\ball#1,#2.{B(#1,#2)} % ball: #1 is the centre, #2 is the radius
\def\clball#1,#2.{\bar B(#1,#2)} % closed ball: #1 is the centre, #2 is the radius
\def\dervec{D_1,\cdots,D_n}
\def\chartfuns{\{x_\alpha^j\}_{j=1}^{N_\alpha}}%set of chart functions
\newcommand{\chartfun}[1][j]{x_\alpha^j}%individual chart function
\newcommand{\cotbund}[1][X]{T^*#1}%cotangent bundle
\def\chart{(X_\alpha,\chartfuns)}%symbol for the charts of a measurable differentiable structure
\def\seccot#1.{{\Gamma(\cotbund[#1])}}%collection of measurable sections of the cotangent bundle
\newcommand{\chartder}[1][]{\frac{\partial #1}{\partial x_\alpha^j}}%partial derivative with respect to a chart function
\def\subalg{\mathcal{A}}%this is the symbol for a subalgebra, it is employed just once in the sequel
\DeclareMathOperator{\rank}{rank}%rank of a matrix
\def\epsi{\varepsilon}%symbol for \varepsilon
\def\rankcoll{\mathcal{R}}%collection of subsets with a given rank, used in the proof of one of the theorems.
\def\rankder{\mathcal{S}_{D'}}%set used in the proof of one of the theorems in the section about the linear algebra of derivations
\def\realproj#1{\mathbb{RP}^{#1}}% symbol for the real projective space
\def\innominato{index} %Kleiner asked to find a better word for rank, I will use a macro till we find it
\def\gensetinfty{\left\{g_j\right\}_{j=1}^\infty}%countable generating set
\def\dual#1.{{{#1}^{\star}}}%dual of a Banach space X
\DeclareMathOperator\sgn{sgn}%symbol for the sign operator

\date{\today}
\title[Derivations and differentiable structures]{On the relationship between derivations and measurable
  differentiable structures on metric measure spaces}
\author{Andrea Schioppa}
\begin{document}
\numberwithin{equation}{section}
\theoremstyle{plain} %to change the headings as suggested by Kleiner
\newtheorem{lem}[equation]{Lemma}
\newtheorem{prop}[equation]{Proposition}
\newtheorem{thm}[equation]{Theorem}
\newtheorem{cor}[equation]{Corollary}
\theoremstyle{definition}
\newtheorem{defn}[equation]{Definition}
\theoremstyle{plain}
\newtheorem*{Free}{Theorem \ref{freemodules}}
\newtheorem*{Finite}{Theorem \ref{der-finite-dimensionality}}
\newtheorem*{Choice}{Theorem \ref{choice}}
\maketitle
\begin{abstract}
  We investigate the relationship between measurable differentiable
structures on doubling metric measure spaces and derivations. We
prove: 
\begin{enumerate}
\item a decomposition theorem for the module of derivations into
free modules;
\item the existence of a measurable differentiable
structure assuming that one can control the pointwise upper Lipschitz constant
of a function through derivations;
\item an extension of a result of
Keith about the choice of chart functions.
\end{enumerate}
\end{abstract}
\section{Introduction}
The extension of first order calculus to metric measure spaces which
are not smooth has
been a topic of research in the last decade. The search for
regularity conditions on a metric measure space allowing to
generalize results and concepts of first order calculus, 
for example the notions of
derivative and gradient, has been a topic of intensive research.
We refer the reader to the survey \cite{heinonen07} for more details.
A fundamental result about the geometry of Lipschitz
functions on Euclidean spaces is the Rademacher Differentiation
 Theorem which asserts that
a Lipschitz function is differentiable a.e.~with
respect to the Lebesgue measure. 
\subsection*{The Rademacher Differentiation Theorem for metric measure
spaces}
In \cite{cheeger99} Cheeger found an
extension of this result to doubling metric measure spaces
which admit a weak version of the Poincar\'e inequality in the sense 
presented by Heinonen and
Koskela in \cite{heinonen98, heinonen_analysis}. The starting point of this generalization is the introduction of a notion of linear
independence of Lipschitz functions at a point (compare Definition \ref{loc_ind_lip}). Because of the
 the Poincar\'e inequality it is possible to prove that there is a uniform bound on
the number of Lipschitz functions that are linearly independent on
a set of positive measure. Such kind of {\bf finite dimensionality}
result can be interpreted as a Rademacher Differentiation Theorem and
used to introduce the notion of a {\bf measurable differentiable
  structure} which allows to take {\bf partial derivatives} with
respect to {\bf chart functions} (see Section \ref{section_mds}).
In \cite{keith04} Keith found a weaker
condition, the  ``Lip-lip'' inequality, which implies the existence of
a measurable differentiable structure. This condition can be
interpreted as a constraint  on the oscillation of a Lipschitz
function at small scales. The oscillation is of course dependent on the scale, but the
``Lip-lip'' inequality prevents a Lipschitz function from oscillating a lot on
some scales and very little on others. For another account of this
result we refer the reader to \cite{kleiner_mackay}. Keith showed also in
\cite{keith04bis} that chart functions can be chosen among distances
from points. His argument assumed a Poincar\'e inequality and
exploited Sobolev spaces techniques.
\subsection*{Derivations}
 It is perhaps surprising that one can introduce a notion of
``derivatives'' on metric measure spaces without requiring much regularity on the metric space
(but the construction can then become trivial). In \cite{weaver00}
Weaver introduced a concept of {\bf derivation} (closely related to
derivations of Banach algebras) which extends the concept 
of a measurable vector field on a (Lipschitz) manifold to a metric
measure space. Cheeger and
Weaver proved that the two constructions agree for the spaces
considered in \cite{cheeger99}: details can be found in \cite[sec.~5,
example F]{weaver00}. However, the relationship between measurable
differentiable structures and derivations is still unclear. We were
motivated to study this relation by the work of Gong \cite{gong11} which
produces bounds on the number of independent derivations on a
doubling metric measure space and recovers a finite dimensionality
result from a ``Lip-derivation'' inequality. 
\subsection*{Main Results}
We summarize here the main results of this work and refer the reader to the corresponding sections for explanations of the terminology. The first result concerns the linear algebra of the derivation module.
\begin{Free}
  Suppose that the module of derivations $\dermodsp$ has \innominato\ locally bounded
  by $N$. Then there is a measurable partition
  \begin{equation*}
    X=X_0\sqcup\cdots\sqcup X_N\sqcup\Omega,
  \end{equation*}
  such that:
  \begin{itemize}
  \item $\mu(\Omega)=0$;
  \item if $X_i\ne\emptyset$ the
    $\elleinfty X_i.$ module $\dermod[X_i]$
  is free of rank $i$.
  \end{itemize}
A basis for $\dermod[X_i]$ will be called a {\bf local basis
  of derivations}.
\end{Free}
The second result is the existence of a measurable differentiable structure assuming that derivations control the pointwise upper Lipschitz constant $\biglip f$ (defined in Sec.~\ref{sec_biglip}) of the function.
\begin{Finite}
  Let $(X,\metric,\meas)$ be a doubling metric measure space.
  Assume that:
  \begin{itemize}
   \item there are $N$ derivations $\dervec$ and a nowhere vanishing
  $\lambda\in\elleinfty X.$;
   \item for any Lipschitz function $f$, there is a set $\Omega_f$
     such that 
  \begin{align*}
    \mu(\Omega_f)&=0;\\ 
    \max_{j=1,\cdots,N}|D_jf(x)|&\ge\lambda(x)\biglip f(x)\quad\forall
    x\in{}^c\Omega_f;
  \end{align*}
  \end{itemize}
    then $X$ admits of a measurable differentiable structure whose dimension
  is at most $N$. 
\end{Finite}
The third result extends the result of \cite{keith04bis} about the choice of the chart functions.
\begin{Choice}
 Suppose the doubling metric measure space
  $(X,\metric,\meas)$ admits a measurable differentiable structure and that
  for each chart  $\chart$ the partial derivatives are derivations.
  If 
  $\genset$ is a generating set for the Lipschitz algebra $\lipalg X.$,
   the charts can be chosen so that the chart functions belong to
  $\genset$.
\end{Choice}
\subsection*{Organization of the paper}
In Section
\ref{sec_lin_algebra}
 we clarify the linear algebraic structure of the module of
derivations. The module
of derivations is a module over ${\rm L}^\infty$. The ring ${\rm
  L}^\infty$ is not an integral domain and some care must be taken in
introducing notions like ``basis'' or ``rank''. We decided to restrict
the term ``basis'' to the case in which the module is free and replace
``rank'' by \innominato\ (so the terminology is different from that used in \cite{gong11}).  In particular, we present a condition to obtain 
 a decomposition of the
module of derivations into free modules by finding a measurable
partition of the metric measure space (Theorem \ref{freemodules}). 
\par In Section \ref{sec_biglip} we recall some  results about the local
Lipschitz constants of a function. We then prove, assuming that the
metric space is doubling, the localized derivation inequality
\eqref{rev-der-ineq} which, roughly speaking, says that if we apply a
derivation $D$ to a Lipschitz function $f$, the size of $Df$ is
locally controlled by the local Lipschitz constant $\biglip f$. 
\par In
Section 4 we recall background material about measurable
differentiable structures. In particular, to a metric measure space
with a measurable differentiable structure it is possible to associate
a measurable cotangent bundle and use this to construct (reflexive)
Sobolev spaces $\introlipsob$ for $p>1$. We choose a different class of
Sobolev spaces from that employed by Cheeger because the minimal upper
gradient might become trivial if the space lacks enough rectifiable
curves. An example to keep in mind is a positive measure Cantor set in
$[0,1]$: it has a measurable differentiable structure as it is a
positive measure subset of $[0,1]$ but the corresponding $\introlipsob$
does {\bf not inject} into the corresponding $\introellep$ space. The
question of injectivity is closely related to the closability of the
exterior differential $d$ coming from the measurable differentiable
structure (Proposition \ref{d:closabilityprop}). 
\par In Section
\ref{sec_fin_dime} we present a finite dimensionality result, i.e.~the
existence of a measurable differentiable structure, assuming that the
metric measure space is doubling and the ``reverse infinitesimal
derivation inequality'' \eqref{derivation-inequality} holds. This condition,
roughly speaking, says that there are sufficiently many derivations to
control the size of the local Lipschitz constant $\biglip f$ up to an
${\rm L}^\infty$ conformal factor $\lambda $ (uniform in the sense that does
not depend on the Lipschitz function). The reverse infinitesimal
derivation inequality should be compared with the ``Lip-derivation''
inequality of \cite{gong11}. Our argument differs from that used by
Gong to prove finite dimensionality as we do not use an embedding into
Euclidean space but we exploit the linear algebraic structure of the
derivation module and the localized derivation inequality. 
\par In Section
\ref{choice_sec} we extend the results of Keith \cite{keith04bis} about
the choice of the chart functions. We first present a representation
formula \eqref{eq:representation} of derivations in terms of partial
derivatives. We then show that if the partial derivatives are
derivations the existence of a measurable differentiable
structure is {\bf equivalent} to the reverse infinitesimal derivation
inequality. We give sufficient conditions for partial derivatives to
be derivations but we are not able to settle the question as Sobolev
space techniques seem insufficient if $\introlipsob$ does not inject in
$\introellep$. We then generalize the result of Keith on the choice of
chart functions (Theorem \ref{choice}) using Lipschitz algebra
techniques. 
\par While writing this note we felt that it might have been
useful to provide some material about Lipschitz algebras and
derivations 
%as we work with an (apparently) different notion of derivations
%from that used in \cite{weaver00}. 
This can be found in the Appendix.
\section{Derivations and Linear Algebra} \label{sec_lin_algebra}
In this section we
first recall the definition and some properties of the Lipschitz
algebra $\lipalg X.$ of a metric space $(X,\metric)$. We then recall
the definition of the $\elleinfty X.$-module $\dermod[X]$ of
derivations of a metric measure space $(X,\metric,\meas)$. Derivations
form a module over the ring of essentially bounded functions. We
proceed to investigate the algebraic structure of this module using
linear algebra and measure theory.  In particular, we give conditions
to decompose the module of derivations into free modules over
``smaller rings'' $\elleinfty U.$ where $U\subset X$ has positive
measure (Theorem \ref{freemodules}). An example to keep in mind is that of smooth vector fields
defined on a smooth manifold $M$. In that case one replaces $\lipalg X.$
 by the algebra of bounded smooth functions and $\dermod[X]$ by the
$C^\infty(M)$-module of smooth vector fields. 
\begin{defn}[Lipschitz Algebra]
  Let $(X,\metric)$ be a metric space. We denote the collection of
  bounded Lipschitz functions on $(X,\metric)$ with values in $\real$
  by $\lipalg X.$. The set $\lipalg X.$ is a real algebra where multiplication
  is defined as follows:
 if $f,g\in\lipalg X.$,
  \begin{equation}
    (fg)(x)=(f(x))(g(x)).
  \end{equation}
\end{defn}
\begin{defn}
  For a Lipschitz function $f:X\to\real$ we denote by $\glip f.$ its
  global Lipschitz constant:
  $$
  \glip f. = \sup_{x\ne y}\frac{|f(x)-f(y)|}{\metric(x,y)}.
  $$
 For $f\in\lipalg X.$ we define the norm
  \begin{equation}
    \lipnorm f! = \supnorm f! \vee \glip f..
  \end{equation}
  This gives $(\lipalg X.,\lipnorm\cdot!)$ the structure of
 a Banach algebra \cite[sec.~4.1]{weaver_book99}. 
\end{defn}
As Weaver points out in \cite[sec.~4.1]{weaver_book99}, the
 term ``Banach algebra'' is used slightly differently in this context as
$\lipalg X.$ is actually bi-Lipschitz to a Banach algebra in the usual
sense. An important property of $\lipalg X.$ is that it is a dual
Banach space and it has a {\bf unique predual}.
The are two approaches to prove this result. The first
approach uses the de Leeuw's map \cite{deleeuw61}. The second approach
gives an explicit description of the dual space in terms of the
Arens-Eells space \cite{arens_eels56}. For more information we refer
the reader to \cite[chap.~2]{weaver_book99} and the Appendix. As $\lipalg
X.$ is a dual Banach space with a unique predual, 
we can consider {\bf the weak* topology} on it. It turns out that
$f_n\to f$ in the weak* topology if and only if $f_n\to f$ uniformly on bounded
subsets and if $\sup_n\glip f_n. <\infty$.
\begin{defn}
  A set $\genset\subset\lipalg X.$ is called a
  {\bf generating set} for $\lipalg X.$ if the subalgebra
  generated by it is weak* dense in $\lipalg X.$.
\end{defn}
An important result connected to the previous definition is
the Stone-Weierstra\ss\ Theorem \ref{stone-weierstrass}.
The following definition of derivations differs from that of Weaver
\cite{weaver00}. However, for a separable metric space 
 the two definitions agree (see the
Appendix). The point is that we require weak* continuity of
derivations just with respect to sequences.
\begin{defn}[Derivations]\label{derivationsdef}
  Let $(X,\metric,\meas)$ be a metric measure space. A map
  \begin{equation}
    D:\lipalg X.\to\elleinfty X.,
  \end{equation}
  is called a {\bf derivation} if:
  \begin{itemize}
  \item is linear and bounded;
  \item satisfies the Leibniz rule
    \begin{equation}
      D(fg)(x)=Df(x)g(x)+f(x)Dg(x);
    \end{equation}
    \item if $f_n\to f$ in the weak* topology in $\lipalg X.$, then
      $Df_n\to Df$ in the weak* topology in $\elleinfty X.$.
  \end{itemize}
  The set of all derivations is denoted by $\dermodsp$ and is an
  $\elleinfty X.$-module. If we restrict
  the Lipschitz functions to a measurable subset of $U\subset X$ we
  denote the $\elleinfty U.$-module of derivations by $\dermod$. % The set $\dermodsp$ is an
  % $\elleinfty X.$ module and $\dermod$ is an $\elleinfty U.$ module.
\end{defn}
 In the subsequent sections we will often assume, to simplify
the notation in the proofs,
 that derivations are defined on bounded spaces with
finite measure. An alternative way would have been to modify the
definition of derivations so that they are defined on the full Lipschitz
algebra $\lipfun X.$ as maps
$$
D:\lipfun X.\to{\rm L}^\infty_{\rm loc}(X,\meas).
$$
For $U\subset X$, $\mu(U)>0$, $\elleinfty U.$ is a commutative ring
with unity. The subset
\begin{equation}
  {\mathcal V}^\infty(U,\meas)=\left\{f\in\elleinfty U.:
  \mu\left(\{x\in U: f(x)=0\}\right)=0\right\} 
\end{equation}
consists of those functions which are {\bf nowhere vanishing}.
We define the set ${\mathcal V}^\infty_M(U,\meas)$ of those functions
whose absolute value is a.e.~bounded from below by $M>0$:
\begin{equation}
  {\mathcal V}^\infty_M(U,\meas)=\left\{f\in\elleinfty U.:
  \mu\left(\{x\in U: |f(x)|<M\}\right)=0\right\}. 
\end{equation}
Note that $f\in\elleinfty U.$ is a unit if and only if
$f\in{\mathcal V}^\infty_M(U,\meas)$ for some $M>0$.
\begin{lem}\label{meas:dep}
  Let $V_1,\cdots,V_M:X\to {\mathcal V}$ be Borel maps, with
  $({\mathcal V},\vnorm\cdot!)$ a normed vector space and $\vnorm
  V_i!\in\elleinfty X.$. Suppose that $\mu(A)>0$ and $\forall x\in A$
  \begin{equation}
    \realspan V_1(x),\cdots,V_M(x).\qquad\text{(real span)}
  \end{equation}
  has dimension $M-1$;
  then there are functions $\lambda_i\in\elleinfty A.$ 
  such that $${\vnorm \lambda_i!}_{\elleinfty A.}\le1$$ and
  \begin{align}
    \sum_{i=1}^M\lambda_i(x)V_i(x)&=0\quad\text{for a.e.~$x\in A$} \\
    \mu\left(\left\{x:\forall i,\lambda_i(x)=0\right\}\right)&=0.
  \end{align}
Furthermore, if  $\forall x\in A$
  \begin{equation}\label{ducentoquindici}
    \realspan V_1(x),\cdots,V_M(x).=\realspan V_1(x),\cdots,V_{M-1}(x).,
  \end{equation}
then $\lambda_{M}\in{\mathcal V}^\infty(A,\meas)$.
\end{lem}
\begin{proof}
  For almost every point $x\in X$ we have  $\left\{\alpha_1(x),\cdots
  \alpha_M(x)\right\}\subset\real$ such that
  \begin{equation}\label{ducentosedici}
 \left(\alpha_1(x),\cdots,
  \alpha_M(x)\right)\ne0,
\end{equation} and
$$
    \sum_{i=1}^M\alpha_i(x)V_i(x)=0.
  $$
Note that if the additional hypothesis \eqref{ducentoquindici} holds,
we can assume that $$\alpha_M(x)\ne0.$$ Once the proof is complete, this
will imply that $\lambda_M$ is nowhere vanishing. 
\par  By assumption any other $M$-tuple such that
\eqref{ducentosedici} holds
is a multilple of
$$
  (\alpha_1(x),\cdots,\alpha_M(x)). 
  $$
Therefore, the map
  \begin{align}
    \Lambda:X&\to\realproj {M-1}\quad\text{(real projective space)}\\
    x&\mapsto[(\alpha_1(x),\cdots,\alpha_M(x))];\\
\end{align}
is well-defined (a.e.~as we need the $\vnorm V_i(x)!$ to be finite).
We define
\begin{align}
    G:\realproj{M-1}\times X&\to\real\\
    (\sigma,x)&\mapsto\max\left\{\|
    \sum_{i=1}^M\alpha_iV_i(x)\|:\text{$(\alpha_i)\in\real^M:
        |(\alpha_i)|\le1$ and $[(\alpha_i)]=\sigma$}\right\};
\end{align}
  then $G$ is continuous in $\sigma$ and Borel measurable in $x$.
  To show that $\Lambda$ is Borel measurable, it suffices to show that
  $\Lambda^{-1}(C)$ is a Borel set whenever $C\subset\realproj{M-1}$ is closed.
  If $\{\alpha_i\}\subset C$ be a countable dense subset, then
  \begin{equation}
    \Lambda^{-1}(C)=\bigcap_n\bigcup_i\left\{x: G(\alpha_i,x)<\frac{1}{n}
      \right\}.
  \end{equation}
  Therefore $\Lambda$ is a Borel function. Finally, $\realproj{M-1}$ can be covered
  by $M$ differentiable charts. 
  On each chart $\Lambda$ can be lifted to
  $(\lambda_1,\cdots,\lambda_M)$ 
  with
  $${\vnorm\lambda_i!}_{\elleinfty X.}\le1.$$
\end{proof}
 The previous argument can be generalized involving the
Grassmanian over the complex or the real fields.
\begin{defn}
  The derivations $\left\{D_1,\cdots,D_n\right\}\subset\dermod$ are
  said to be {\bf linearly independent} (over $\elleinfty U.$)
  if for any
  $\left\{\lambda_1,\cdots,\lambda_n\right\}\subset\elleinfty
  U.$,
  \begin{equation}
    \sum_{i=1}^n\lambda_i D_i=0,
  \end{equation}
implies that $\lambda_i=0$. This means that for any choice of the
representatives for the $\lambda_i$, these vanish a.e. 
In the sequel, we have 
not kept the distinction between elements of ${\rm L}^p$-spaces and their
representatives. 
\end{defn}
\begin{defn}[Finite \innominato]
If in $\dermod$
any linearly independent set of derivations has at most $N$ elements,
 $\dermod$ is said to have finite \innominato. The smallest value of $N$ is
 the \innominato\ of
$\dermod$.
\end{defn}
 The previous definition of \innominato\ is an
attempt to generalize the notion of the rank of
a free module. If $\dermod$ were free, rank and \innominato\ would
agree. We now prove an implication of the condition that $\dermod[U]$
has finite \innominato. We first assume that $\lipalg X.$ has a finite
generating set $\genset$ ($M<\infty$). We will then reformulate this
result for the case in which $\genset$ is countable.
\begin{prop}\label{ducentoventidue}
  Suppose that $\{D_i\}_{i=1}^n\subset\dermod$ is a linearly
  independent set.
Suppose $\lipalg U.$ has a finite generating set $\genset$.  Then for
  a.e.~$x\in U$ the 
  row vectors 
  \begin{equation}
    D_ig(x)\equiv (D_ig_1(x),\cdots,D_ig_M(x))
  \end{equation}
are linearly independent. In particular, $M\ge n$.
\end{prop}
\begin{proof}
  Let $V_i=D_ig$. Suppose that
  there is a measurable $V\subset U$ such that $\meas(V)>0$ and
  $\forall x\in V$
  $$
    \realspan V_1(x),\cdots,V_n(x).
$$ has dimension strictly smaller than $n$. Without loss of generality
we can assume that $\forall x\in V$
$$
\realspan V_1(x),\cdots,V_{k}(x).
$$ has dimension $k-1$. We now apply Lemma \ref{meas:dep} with
$\mathcal{V}=\real^M$
obtaining
$$\left\{\lambda_1,\cdots,\lambda_k\right\}\subset\elleinfty V.
$$ with   \begin{align}
    \sum_{i=1}^{k}\lambda_i(x)V_i(x)&=0 \\
    \mu\left(\left\{x:\forall i,\lambda_i(x)=0\right\}\right)&=0.
  \end{align}
We extend the $\lambda_i$ to $\elleinfty U.$ by setting them equal to
$0$ on $U\setminus V$. The conclusion is that the derivation
$$
D'=\sum_{i=1}^k\lambda_iD_i
$$ maps each $g_i$ to $0$. By the Leibniz rule, $D'=0$ on the algebra
generated by $\genset$. But by weak* continuity,the derivation $D'$ is
trivial on $\lipalg U.$. Therefore the derivations $D_1,\cdots,D_k$
are linearly depedent, contrary to the hypothesis. Note also that as
the $V_i$ have to be linearly independent a.e., $M\ge n$.
\end{proof}\begin{cor}\label{orthogonality}
    Suppose that $\{D_i\}_{i=1}^n\subset\dermod$ is a linear
    independent set.  Let
    $\genset$ be a finite generating set for $\lipalg U.$. Then there are:
      \begin{itemize}
\item a subset of the generators $\{g'_1,\cdots,g'_n\}\subset\genset$;
 \item a measurable $V\subset
  U$ with $\meas(V)>0$
\item
  an invertible $n\times n$ matrix $A$ over the ring $\elleinfty V.$, that is,
  $A=(a_{ij})_{i,j=1,\cdots,n}$ for $a_{ij}\in\elleinfty V.$ such that if
  we define a new set of derivations
   $\{D'_i\}_{i=1}^n\subset\dermod$ by \begin{equation} A\,D\equiv D',
\end{equation}
then $\forall x\in V$
  \begin{equation}\label{ducentotrentuno}
    D'_ig'_j(x)=\delta_{i,j}.
  \end{equation}
      \end{itemize}
\end{cor}
\begin{proof}
  By Proposition \ref{ducentoventidue} for a.e.~$x\in U$ the $D_ig(x)$
  are linearly independent. Therefore there are
  \begin{itemize}
  \item a measurable $U'\subset U$ with $\meas(U')>0$ and
    \item a subset $\{g'_1,\cdots,g'_n\}\subset\genset$,
  \end{itemize}
such that the matrix 
$$
B(x)=    \begin{pmatrix}
      D_1g'_1(x) & \cdots & D_1g'_n(x) \\
      \vdots & \vdots & \vdots \\
      D_ng'_1(x) & \cdots & D_ng'_n(x) \\
    \end{pmatrix}
$$ is non singular for each $x\in U'$. In particular, $\det B\ne0$ on
$U'$ so we can find $\epsi>0$ and 
 $V\subset U'$ with $\meas(V)>0$ and $|\det B|>\epsi$
 on $V$. As $\det B\in {\mathcal V}^\infty_\epsi(V,\meas)$, we have that
 \begin{equation} f(x)=
   \frac{1}{\det B(x)}\in\elleinfty V..
\end{equation} 
If we let $C$ be the cofactor matrix of
$B$ and define new derivations in $\dermod[V]$ by
$$
D'_i=\sum_{j=1}^n f C_{ij} D_j|_V,
$$ we have that \eqref{ducentotrentuno} holds a.e.~in $V$. We finally
let $A=fC$.
\end{proof}
We now discuss the modifications for the case $M=\infty$ by which we mean
that $\genset$ is countable. This assumption is really mild. In fact,
if $X$ is separable, then by the Stone-Weierstra\ss\ Theorem (\ref{stone-weierstrass}) $\lipalg X.$ has a countable generating set. The main
point is to find a normed vector space in which the vectors $D_ig$ lie.
\begin{prop}\label{ducentoventiduebis}
  Suppose that $\{D_i\}_{i=1}^n\subset\dermod$ is a linearly
  independent set.
  Let $\gensetinfty$ be a countable generating set for $\lipalg U.$ with
  \begin{equation}
    \sup_{j=1,\cdots,\infty}\lipnorm g_j! \le C<\infty.
  \end{equation}
  Then for
  a.e.~$x\in U$ the 
  row vectors 
  \begin{equation}
    D_ig(x)\equiv (D_ig_j(x))_{j=1}^\infty\in l^\infty(\mathbb{N})
  \end{equation}
are linearly independent. Furthermore there are
$\{g'_1,\cdots,g'_n\}$, $V$ and $A$ such that the conclusions of
Corollary \ref{orthogonality} hold.
\end{prop}
\begin{proof}
  The proof of the first part is like that of Proposition \ref{ducentoventidue} but we
  take $\mathcal{V}=l^\infty(\mathbb{N})$. For the proof of the second
  part we could argue as in Corollary \ref{orthogonality} provided
  that there are a subset of the generators $\{g'_1,\cdots,g'_n\}$ and
  a set $U'$ of positive measure such that $$ B(x)=    \begin{pmatrix}
      D_1g'_1(x) & \cdots & D_1g'_n(x) \\
      \vdots & \vdots & \vdots \\
      D_ng'_1(x) & \cdots & D_ng'_n(x) \\
    \end{pmatrix} $$ is nonsingular on $U'$. We prove this arguing by
    contradiction. Let $T_M:l^\infty(\mathbb{N})\to l
    ^\infty(\mathbb{N})$ be the truncation map:
    \begin{equation}
      T_M(c_i)_{i=1}^\infty=(c_1,c_2,\cdots,c_M,0,0,\cdots).
    \end{equation} If we cannot find a subset $U'\subset U$ and a
    subset of the generators such that $B$ is nonsingular on $U'$,
    then for each $M$ the vectors $T_M(D_ig)$ are linearly dependent
    a.e. This implies that there is a subset $\tilde U\subset U$ with 
    $\mu(U\setminus \tilde U)=0$ and for each $M\ge1$ and $x\in\tilde
    U$, the vector subspace
    \begin{equation}
    \Lambda_M(x)=\left\{(\lambda_1,\cdots,\lambda_n)\in\mathbb{R}^n:
      \sum_{i=1}^n\lambda_iT_M(D_ig)(x)=0\right\}
  \end{equation}
has dimension at least $1$. Now, $\Lambda_M(x)\supset\Lambda_{M+1}(x)$ and
so
\begin{equation}\label{ducentoventidueventiduebis}
  \bigcap_{M=1}^\infty \Lambda_M(x)\ne\{0\}
\end{equation}
as the dimension of a vector subspace of $\mathbb{R}^n$ has to lie in
$\{ 0,1,\cdots,n\}$ and the dimension of $\bigcap_{M=1}^k\Lambda_M(x)$ can decrease
only by integer values as we increase $k$. But then
\eqref{ducentoventidueventiduebis} would imply that the $D_ig(x)$ are
linearly dependent.
\end{proof}
 Note that in Proposition \ref{ducentoventidue} the role of generators
and derivations is not symmetrical. For example, we can have $M>n$
(trivially
if $M=\infty$ as in Proposition \ref{ducentoventiduebis}). An
example is the standard Cantor set whose Lipschitz algebra is
generated by the single function $x$ and where all derivations are
trivial (see \cite[Section 5.~Examples A]{weaver00}
 or use the localized derivation
inequality \eqref{rev-der-ineq} which is proved in the next
section). Note that the bound $M>n$ for $M$ finite implies the hypothesis of
Theorem \ref{freemodules}.
\begin{defn}[Modules of derivations whose \innominato\ is locally
bounded] We say that the module of derivations $\dermodsp$ has
{\bf \innominato\ locally bounded} by $N$ if for any set $U\subset X$
of positive measure, $\dermod$ has \innominato\ at most $N$ over
$\elleinfty U.$.
\end{defn}
\begin{thm}\label{freemodules}
  Suppose that the module of derivations $\dermodsp$ has \innominato\ locally bounded
  by $N$. Then there is a measurable partition
  \begin{equation}
    X=X_0\sqcup\cdots\sqcup X_N\sqcup\Omega,
  \end{equation}
  such that:
  \begin{itemize}
  \item $\mu(\Omega)=0$;
  \item if $X_i\ne\emptyset$ the
    $\elleinfty X_i.$ module $\dermod[X_i]$
  is free of rank $i$.
  \end{itemize}
A basis for $\dermod[X_i]$ will be called a {\bf local basis
  of derivations}.
\end{thm}
\begin{proof}
  Without loss of generality we assume that $\meas(X)<\infty$. Let
  $\rankcoll_N(X)$ denote the collection of subsets $U\subset X$
  satisfying the following properties:
  \begin{itemize}
  \item $U$ is measurable and $\meas(U)>0$;
    \item $\dermod$ has \innominato\ $N$;
  \end{itemize}
if $\rankcoll_N(X)=\emptyset$ we let $X_N=\emptyset$. If
$\rankcoll_N(X)\ne\emptyset$, let $U_{N,1}\in\rankcoll_N(X)$ be such
that
$$
   \meas(U_{N,1})>\frac{2}{3}\sup_{V\in\rankcoll_N(X)}\meas(V).
$$
If $$
\rankcoll_N(X\setminus U_{N,1})=\emptyset,
$$ we stop; otherwise we select $U_{N,2}\in\rankcoll_N(X\setminus
U_{N,1})$
with
$$
   \meas(U_{N,2})>\frac{2}{3}\sup_{V\in\rankcoll_N(X\setminus
     U_{N,1})}\meas(V).
$$ The construction of the sets $\left\{U_{N,i}\right\}$ proceeds by
induction. There are two cases: either we stop after $N'$ steps or we
continue up to infinity. In the first case we let
$$
X_N=U_{N,1}\sqcup\cdots\sqcup U_{N,N'}
$$ and observe that 
$$
\rankcoll_N(X\setminus X_N)=\emptyset.
$$
In the second case, as the sets $\left\{U_{N,i}\right\}$ are disjoint
and as $\meas(X)<\infty$, we conclude that 
$$
\lim_{i\to\infty}\meas(U_{N,i})=0.
$$ We now observe that
\begin{equation}
  \begin{split}
\sup_{V\in\rankcoll_N\left(X\setminus\bigsqcup_{i=1}^\infty
    U_{N,i}\right)}
\meas(V)&\le\sup_{V\in\rankcoll_N\left(X\setminus\bigsqcup_{i=1}^k
    U_{N,i}\right)}\meas(V)\\
&\le\frac{3}{2}\meas(U_{N,k+1})\to0
  \end{split}
\end{equation}
as $k\nearrow\infty$; this shows that
$$
\rankcoll_N(X\setminus\cup_i U_{N,i})=\emptyset.
$$ 
The conclusion is that if we let $X_N=\sqcup_{i=1}^\infty U_{N,i}$
then there is no measurable subset $V\subset X\setminus X_N$ such that 
$\mu(V)>0$ and $\dermod[V]$ has \innominato\ $N$. The sets
$X_{N-1},X_{N-2},\cdots,X_0$ are constructed by induction. Here we use
the hypothesis that the \innominato\ is locally bounded by $N$. For example,
this implies that $X\setminus X_N$ has \innominato\ locally bounded by $N-1$.
The induction step proceeds as follows. Suppose we have already
constructed $X_N,\cdots,X_{N-k}$ and $N-k>0$. If we let
$$
Y=X\setminus \left(X_{N-k}\sqcup\cdots\sqcup X_N\right),
$$ then $Y$ has \innominato\ locally bounded by $N-k-1$. Let
  $\rankcoll_{N-k-1}(Y)$ denote the collection of subsets $U\subset Y$
  satisfying the following properties:
  \begin{itemize}
  \item $U$ is measurable and $\meas(U)>0$;
    \item $\dermod$ has \innominato\ $N-k-1$;
  \end{itemize}
then we apply the same argument used to construct $X_N$. In
particular, there is no measurable subset $V\subset Y\setminus
X_{N-k-1}$ such that $\meas(V)>0$ and $\dermod[V]$ has \innominato\ $N-k-1$.
Before proceeding further, we remark that $X_0$ might be nonempty.
We now show that 
$$
\meas\left(X\setminus (X_0\sqcup\cdots\sqcup X_N)\right)=0
$$ arguing by contradiction. If
$$
\meas\left(X\setminus(X_0\sqcup\cdots\sqcup X_N)\right)>0,
$$ there would be some measurable
$$
V \subset X\setminus X_0\sqcup\cdots\sqcup X_N
$$ with $\meas(V)>0$ and $\dermod[V]$ having \innominato\  in
$\left\{0,\cdots,N\right\}$.  This contradicts the construction
of the $X_i$'s. Note that in this step we have again used that the
\innominato\  is locally bounded by $N$. 
\par We now prove that $\dermod[X_k]$ is free over $\elleinfty X_k.$
. We choose  a
maximal linearly independent set
$$\left\{D_1,\cdots,D_k\right\}\subset\dermod[X_k]$$ and we show it is a basis. As the
elements of this set
are linearly independent, it suffices to show that it spans
$\dermod[X_k]$ over $\elleinfty X_k.$. Let $D'\in\dermod[X_k]$ and
define the collection $\rankder(X_k)$ of subsets of $X_k$ in the following
way:
$V\in\rankder(X_k)$ if and only if the following holds:
\begin{itemize}
\item $V$ is measurable and $\meas(V)>0$;
\item there are
  $\left\{\lambda_{1,V},\cdots,\lambda_{k,V}\right\}\subset\elleinfty
  X_k.$
such that:
\begin{equation}\label{ducentoventifa}
\chi_V\cdot\left(D'-\sum_{j=1}^k\lambda_{j,V}D_j\right)=0.
\end{equation}
\end{itemize}
We first show that $\rankder(X_k)$ is not empty. By the maximality of
$\left\{D_1,\cdots,D_k\right\}$ there are
$\left\{\lambda',\lambda_1,\cdots,\lambda_k\right\}\subset\elleinfty X_k.$
such that
\begin{align}
  \meas\left(x:
 \lambda'(x)=0\right)&=0\\
\lambda' D' - \sum_{j=1}^k\lambda_jD_j&=0;
\end{align}
Let $V$ be a measurable set with $\mu(V)>0$ and $\lambda'|_V\in   {\mathcal
  V}_M^\infty( V,\meas)$ for some $M>0$, implying that
  $\lambda'|_V$ is invertible in $\elleinfty
V.$. Without loss of generality we can assume that $|\lambda'(x)|\ge M$
for $x\in V$. If we define
$$
\lambda_{j,V}=
\begin{cases}
  \frac{\lambda_j(x)}{\lambda'(x)}&\text{for $x\in V$}\\
  0&\text{for $x\in{}^cV$}\\
\end{cases}
$$ then \eqref{ducentoventifa} holds implying that
$V\in\rankder(X_k)$.
The same argument used for the set $\rankcoll_N(X)$ shows that there
is a measurable partition
% Let $U_1\in\rankder(X_k)$ be such that 
% $$
%    \meas(U_1)>\frac{2}{3}\sup_{V\in\rankder(X_k)}\meas(V).
% $$ If $\meas(X_k\setminus U_1)=0$ we stop, otherwise, as
% $\mu(X_k\setminus U_1)>0$ we conclude that $\rankder(X_k\setminus
% U_1)\ne\emptyset$ (here we use again that for any subset $V\subset
% X_k$ of positive measure, $\dermod[V]$ has \innominato\  $k$). Let $U_2\subset
% X_k\setminus U_1$ be such that
% $$
%    \meas(U_2)>\frac{2}{3}\sup_{V\in\rankder(X_k\setminus U_1)}\meas(V).
% $$ The construction of the sets $\left\{U_i\right\}$ proceeds by
% induction. Either we stop after finitely many steps, for example after
% $N'$ steps or we continue up to $\infty$. In the first case
% $$
% \meas\left(X_k\setminus U_1\cup\cdots\cup U_{N'}\right)=0;
% $$ in the second case as the sets $\left\{U_i\right\}$ are disjoint
% and $\meas(X_k)<\infty$, we conclude that
% $$
% \lim_{i\to\infty}\meas(U_i)=0.
% $$ We now observe that
% \begin{equation}
%   \begin{split}
% \sup_{V\in\rankder\left(X_k\setminus\bigcup_{i=1}^\infty
%     U_i\right)}
% \meas(V)&\le\sup_{V\in\rankder\left(X_k\setminus\bigcup_{i=1}^k
%     U_i\right)}\meas(V)\\
% &\le\frac{3}{2}\meas(U_{k+1})\to0
%   \end{split}
% \end{equation}
% as $k\nearrow\infty$; this shows that
$$
\bigsqcup_{i\in I}
    U_i 
$$ of $X_k$ such that 
% of $X_k$ with
% $$
% \sum_i\chi_{U_i}=1\quad\text{in $\elleinfty X_k.$}.
% $$
%Now 
for each $U_i$ there are $\left\{\lambda_{1,U_i},\cdots
\lambda_{k,U_i}\right\}$ with
\begin{equation}\label{ducentoventibo}
\chi_{U_i}\cdot\left(D'-\sum_{j=1}^k\lambda_{j,U_i}D_j\right)=0;
\end{equation} if we let
$$
\lambda_j=\sum_{i\in I}\chi_{U_i}\lambda_{j,U_i}
$$ and sum the equations \eqref{ducentoventibo}
we conclude that
$$
D'=\sum_{j=1}^k\lambda_jD_j.
$$
\end{proof}
\section{The local Lipschitz constants}\label{sec_biglip}
 In this section we recall the definition of the local Lipschitz
constants $\biglip f$ and $\smllip f$ for a function $f$. Using Egorov
and Lusin Theorems,
we obtain a measurable decomposition where the local Lipschitz
constants behave nicely. In a doubling metric measure space, using
the Lebesgue Differentiation Theorem and the locality principle for
derivations, we obtain the local estimate \eqref{rev-der-ineq} which
we call the {\bf localized derivation inequality}.
\begin{defn}\label{biglip-defn}[Variation and local Lipschitz constants]
Let $f$ be a Lipschitz function. Let us define the {\bf variation} of $f$ on
the ball $B(x,r)$ by
\begin{equation}
  \varlip f(x,r)=\frac{1}{r}\sup_{y\in B(x,r)}|f(x)-f(y)|.
\end{equation}
We define the {\bf lower and upper variations of $f$ 
at $x$ from scale $r$ down to $0$} by
\begin{align}
 \biglip f(x,r)&=\sup_{s\le r}\varlip f(x,s)\\
 \smllip f(x,r)&=\inf_{s\le r}\varlip f(x,s).
\end{align}
Let us define {\bf the infinitesimal Lipschitz constants of $f$ at
$x$} by
\begin{align}
  \biglip f(x)&=\inf_{r\ge0}\biglip f(x,r)\\
  \smllip f(x)&=\sup_{r\ge0}\smllip f(x,r).
\end{align}
\end{defn}
As far a we understand, the behaviour of $\smllip f$ is not so nice in
general. Also, this is not really a local Lipschitz constant. The
behaviour of $\biglip f$ is more regular. For example, if we blow up
$f$ near some point, then $\biglip f$ really gets closer and closer to
the Lipschitz constant of the blow up. We also note that $\biglip f$
behaves like a seminorm in $f$ in the following sense, if $f,g$ are
Lipschitz functions and $\lambda,\mu\in \mathbb{R}$, then
\begin{equation}
  \biglip (\lambda f + \mu g) \le |\lambda|\biglip f + |\mu|\biglip g.
\end{equation} For the following Lemma compare \cite{kleiner_mackay}:
\begin{lem}\label{biglip-meas}
  Let $f$ be a Lipschitz function. Then there is a measurable partition
  \begin{equation}
    X=\bigsqcup_{i=1}^\infty A_i \sqcup \Omega,
  \end{equation} such that
  \begin{itemize}
   \item $\Omega$ has measure $0$
   \item $\biglip f$ and $\smllip f$ are continuous on each $A_i$
\item $\biglip f(\cdot,r)\searrow\biglip f$ and $\smllip
  f(\cdot,r)\nearrow\smllip f$ uniformly on each $A_i$ for $r\searrow0$.
\end{itemize}
\end{lem}
\begin{proof}
  Without loss of generality we can assume $\mu(X)<\infty$. The proof
  uses Egorov and Lusin theorems. It is therefore necessary to
  establish that $\biglip f$ and $\smllip f$ are (Borel)
  measurable. We first observe that, for fixed $s$, $\varlip f(x,s)$
  is lower semicontinuous. Indeed, if $\varlip f(x,s)>C$, there is a
  point $y_s\in\ball x,s.$ such that
$$
\left|f(x)-f(y_s)\right|>C;
$$ then for 
$$
\dist x,x'.<s-\dist x,y_s.
$$  it follows that $y_s\in\ball x',s.$ which implies $\varlip f(x',s)>C$. In
particular, $\varlip f(x,s)$ is Borel (in $x$). This implies that
$\biglip f(x,s)$ and $\smllip f(x,s)$ are Borel (in $x$). As
$\biglip f(x,r)\searrow \biglip f(x)$ and $\smllip f(x,r)\nearrow
\smllip f(x)$, the functions $\biglip f$ and $\smllip f$ are Borel too. By Lusin
theorem, for each $\epsi>0$ there is a subset $B_\epsi\subset X$ such
that
$$
\mu\left(X\setminus B_\epsi\right)<\epsi
$$ and the functions $\biglip f$, $\smllip f$ are continuous on $B_\epsi$. Using an
exhaustion argument we get a measurable partition
$$
X=\bigsqcup_{i=1}^\infty B_i\sqcup \Omega_1
$$ such that $\mu(\Omega_1)=0$ and the functios $\biglip f$, $\smllip f$ are
continuous on each $B_i$. We now consider an arbitrary
 sequence $r_n\searrow0$. By
Egorov theorem, for each $\epsi>0$ and $i$ there is a $C_\epsi\subset
B_i$
such that
$$
\mu\left(B_i\setminus C_\epsi\right)<\epsi
$$ and the sequences $\biglip f(\cdot,r_n)$, $\smllip f(\cdot, r_n)$ converge
uniformly on $C_\epsi$. Using an exhaustion argument we find a
measurable partition
$$
B_i=\bigsqcup_{j=1}^\infty C_{i,j}\sqcup \Omega_{2,i}
$$ such that $\mu(\Omega_{2,i})=0$, and the sequences $\biglip f(\cdot,r_n)$, $\smllip f(\cdot, r_n)$ converge
uniformly on $C_{i,j}$. As $\biglip f(\cdot,r)$ is nonincreasing in
$r$ and $\smllip f(\cdot,r)$ is nondecreasing in $r$, this implies
that $\biglip f(\cdot,r)\searrow \biglip f(\cdot)$ and 
$\smllip f(\cdot,r)\nearrow \smllip f(\cdot)$ uniformly on each
$C_{i,j}$.
The proof is completed by letting
$$
\Omega=\Omega_1\cup\bigcup_{i=1}^\infty\Omega_{2,i}.
$$
\end{proof}
The following discussion is not actually needed to prove Theorem
\ref{rev-der-ineq}, which is the main result of this section. However,
it clarifies the point we made when we said that $\biglip f(x)$
is essentially the Lipschitz constant of $f$ in a neighbourhood of
$x$. Note that we assume that the measure $\meas$ is doubling; 
however, the proof just requires the Lebesgue Differentiation Theorem.
\begin{defn}[local density]
  Let $(X,\metric)$ be a metric space, $A\subset X$ and $x\in X$. We say that $A$ is
  {\bf locally dense at $x$} if for any $\varepsilon>0$ there is an $r(\varepsilon)
>0$ such that if $r\le r(\varepsilon)$, $A\cap \ball x,(1+\epsi)r.$ is $\varepsilon r$-dense in $\ball x,r.$.
\end{defn}

\begin{prop} \label{local-density}
  Let $(X,\metric,\meas)$ be a doubling metric measure space and
$A\subset X$ a measurable subset. Then for a.e.~$x\in A$, $A$ is locally
 dense at $x$.
\end{prop}
\begin{proof}The case $\mu(A)=0$ is trivial so we assume that $A$ has
  positive measure.
  As $\mu$ is doubling, there are constants $C\ge1$ and $\kappa>0$
  such that if $z,w\in X$ and $\ball w,s.\subset \ball x,r.$, we have
  \begin{equation}\label{trecentotririci}
    \frac{\mu(\ball w,s.)}{\mu(\ball x,r.)}\ge\frac{1}{C}\left(
      \frac{s}{r}\right)^\kappa;
  \end{equation}
let $x\in A$ and suppose that $A\cap \ball x,(1+\epsi)r.$ is not
$\epsi r$-dense in $\ball x,r$. In this case, there is a point
$y\in\ball x,r.$ such that $\ball y,\epsi r.$ is disjoint from $$A\cap
\ball x,(1+\epsi)r..$$ As $\ball y,\epsi r.\subset\ball x,(1+\epsi)r.$,
\eqref{trecentotririci} implies that 
\begin{equation}
  \label{eq:trequatto}
  \frac{\mu\left(\ball x,(1+\epsi)r.\setminus A\right)}{\mu\left(
    \ball x,(1+\epsi)r.\right)}\ge\frac{\mu\left(\ball y,\epsi r.\right)}{\mu\left(
    \ball x,(1+\epsi)r.\right)}\ge\frac{1}{C}\left(
      \frac{\epsi}{1+\epsi}\right)^\kappa;
\end{equation}
as the Lebesgue differentiation theorem holds in the metric measure 
space $(X,\metric,\meas)$, for a.e.~$x\in A$, $x$ is a density point
of $A$, that is,
$$
\lim_{s\searrow0}\frac{\mu(\ball x,s.\setminus A)}{\mu(\ball x,s.)}=0.
$$ If we choose $s_\epsi$ so that $r\le s_\epsi$ implies
$$
\frac{\mu(\ball x,(1+\epsi)r.\setminus A)}{\mu(\ball x,(1+\epsi)r.)}<\frac{1}{C}\left(\frac{\epsi}{1+\epsi}\right)^\kappa,
$$
then \eqref{eq:trequatto} does not hold, impliying that $A\cap \ball x,(1+\epsi)r.$ is 
$\epsi r$-dense in $\ball x,r$.
\end{proof}
\begin{cor}
  Suppose $x_0\in A$ is a density point of $A$. Suppose that $f,g$ are
  Lipschitz function with $f=g$ on $A$. Then
  \begin{align}
    \biglip f(x_0)&=\biglip g(x_0)\\
    \smllip f(x_0)&=\smllip g(x_0).\\
  \end{align}
\end{cor}
 We now come back to the proof of Theorem \ref{rever-der-ineq}. We
need a result about the local behaviour of derivations (see
 the Appendix). References are \cite[Lemma
13.4]{heinonen07}, \cite[Lemma 7.2.3]{weaver_book99} and \cite[Lemma
27]{weaver00}. Note that Proposition \ref{locality} does not require
the measure to be doubling or any particular regularity. In our view,
it is essentially a consequence of the functional-analytic properties
of derivations.
\begin{prop}\label{locality}
  Let $A\subset X$ be a measurable set and $D\in\dermodsp$ a derivation. If the Lipschitz
  functions $f,g$ agree on $A$, then $Df(x)=Dg(x)$ for a.e.~$x\in A$.
\end{prop}
We can now prove the main result of this section. Note that 
 we assume that the measure $\meas$ is  doubling,
but the proof
just requires the Lebesgue Differentiation Theorem.
\begin{thm}\label{rever-der-ineq}
  Let $D\in\dermodsp$ and $f\in\lipalg X.$. Assume that the measure
  $\mu$ is doubling. Then there is a measurable set $\Omega_f$ such
  that
  \begin{itemize}
  \item $\mu(\Omega_f)=0$;
    \item if $x\in{}^c\Omega_f$,
        \begin{equation}\label{rev-der-ineq}
    |Df(x)|\le \vnorm D! \biglip f(x).
  \end{equation}
  \end{itemize}
  We will refer to this as the {\bf localized derivation inequality}.
\end{thm}
\begin{proof}
  Without loss of generality we assume that $$
  \vnorm D! \le 1.$$
  We apply Lemma \ref{biglip-meas} to $f$ obtaining a measurable
  partition   $$
    X=\bigsqcup_{i=1}^\infty A_i \sqcup \Omega,
  $$ and it sufficies to show that \eqref{rev-der-ineq} holds for
  a.e.~$x\in A_i$. We will prove that \eqref{rev-der-ineq} holds if
  \begin{itemize}
  \item $x$ is a Lebesgue point of $Df$ and
    \item $x$ is a density point of $A_i$.
  \end{itemize}
As $\biglip f$ is continuous on $A_i$, it follows that for each
$\epsi>0$ there
is an $r_0(x,\epsi)>0$ such that if $r\le r_0(x,\epsi)$ and $y\in\ball
x,2r.$, then 
\begin{equation}
  \label{eq:b}
  \biglip f(y)\le \biglip f(x)+\epsi;
\end{equation}
as $\biglip f(\cdot,r)\searrow \biglip f(\cdot)$ uniformly on $A_i$,
it follows that for each $\epsi>0$ there
is an $r_1(x,\epsi)>0$ such that if $r\le r_1(x,\epsi)\le r_0(x,\epsi)$,
then
\begin{equation}
  \label{eq:a}
  \biglip f(y,2r)\le\biglip f(y)+\epsi.
\end{equation}
We now claim that for $r\le r_1(x,\epsi)$ the restriction
$f|_{A_i\cap\ball x,r.}$  has Lipschitz
constant $\biglip f(x)+2\epsi$. To verify the claim, let
$y_1,y_2\in\ball x,r.$. Then $y_2\in\ball y_1,2r.$ and from the
definition of $\biglip f(y,2r)$ we conclude that
$$
\left|f(y_1)-f(y_2)\right|\le\biglip f(y,2r)\dist y_1,y_2.;
$$ but by \eqref{eq:a} we conclude that
$$
\left|f(y_1)-f(y_2)\right|\le(\biglip f(y)+\epsi)\dist y_1,y_2.,
$$ and by \eqref{eq:b} this gives
$$
\left|f(y_1)-f(y_2)\right|\le(\biglip f(x)+2\epsi)\dist y_1,y_2.,
$$
verifying the claim. 
\par We now note that 
$$
(f-f(x))|_{A_i\cap\ball x,r.}
$$ has Lipschitz constant at most $\biglip f(x)+2\epsi$ and that
$$
\supnorm (f-f(x))|_{A_i\cap\ball x,r.}!\le(\biglip f(x)+2\epsi)r;
$$ we can therefore take a MacShane extension $g$ of
$(f-f(x))|_{A_i\cap\ball x,r.}$ with
$$
\lipnorm g!\le\biglip f(x)+\epsi.
$$ We want to bound
$$
\avint_{\ball x,r.}Df(y)\,d\mu(y)
$$ as $r\searrow0$, because this will give an upper bound for
$|Df(x)|$. The Leibniz rule implies that $D(1)=0$ a.e., therefore
$$
\avint_{\ball x,r.}Df(y)\,d\mu(y)=\avint_{\ball
  x,r.}D(f-f(x))(y)\,d\mu(y);
$$ by Proposition \ref{locality} $D(f-f(x))=Dg$ a.e.~in $A_i\cap\ball
x,r.$ and moreover, as $\vnorm D!\le 1$,
$$
\left| Dg\right|\le\biglip f(x)+2\epsi
$$ a.e. Therefore,
\begin{equation}
\begin{split}
  \left|\avint_{\ball x,r.}Df(y)\,d\mu(y)\right|&=
  \left|\avint_{\ball x,r.}D(f-f(x))(y)\,d\mu(y)\right|\\
&\le\frac{1}{\mu(\ball x,r.)}\int_{A_i\cap\ball
  x,r.}|Dg(y)|\,d\mu(y)\\
&+\frac{1}{\mu(\ball x,r.)}\int_{\ball
  x,r.\setminus A_i}|D(f-f(x))(y)|\,d\mu(y)\\
&\le\biglip f(x)+2\epsi+2\lipnorm f!\frac{\mu\left(\ball x,r.\setminus
    A_i\right)}{\mu\left(\ball x,r.\right)};
\end{split}
\end{equation} letting $r\searrow0$ gives the bound
$$
|Df(x)|\le\biglip f(x)+2\epsi.
$$
\end{proof}
\section{Measurable differentiable structures}\label{section_mds}
 In this section we recall the definition of measurable
differentiable structure. In order to make the exposition more
transparent, we decided to first introduce a notion of local independence
for Lipschitz functions and, building on this definition,
recall Lemma \ref{finite-dimensionality} which implies
 the existence of measurable differentiable
structures. This Lemma has been either explicitly or implicitly used
in previous proofs that a metric measure space admits a measurable
differentiable structure \cite{kleiner_mackay}, \cite{keith04} and \cite[Section 4]{cheeger99}.
The definition of local independence makes also precise the
intuitive idea that, if the a space has a differentiable structure,
the Lipschitz functions form, infinitesimally, a finite dimensional
vector space.
To a space possessing a measurable differentiable structure it is possible
to associate a natural measurable cotangent bundle. Using the sections
of this bundle it is possible to construct Sobolev spaces which are
reflexive for $p>1$. In this setting the exterior derivative $d$
extends to Sobolev functions. We finally make an observation relating
$d$ and the property that these Sobolev spaces inject into the
corresponding ${\rm L}^p$ spaces.
\begin{defn}[Measurable Differentiable Structure]
A metric measure space $(X,\metric,\meas)$ 
  has a {\bf measurable differentiable structure} if:
  \begin{itemize}
    \item there is a measurable decomposition
      \begin{equation}X=\bigcup_{\alpha}X_\alpha\cup\Omega;\end{equation}
    \item $\mu(\Omega)=0$;
    \item for each set $X_\alpha$ there are Lipschitz functions
      $\chartfuns$ such that if $f$ is a Lipschitz function,
      there are unique $\elleinfty X_\alpha.$ functions:
      \begin{equation}
        \frac{\partial f}{\partial x_\alpha^j}:X_\alpha\to \real
      \end{equation}
      such that 
      \begin{equation}
        \biglip\left\{f-\sum_{j=1}^{N_\alpha}
            \frac{\partial f}{\partial x_\alpha^j}(x)\chartfun\right\}(x)=0,
      \end{equation}
      for a.e.~$x\in X_\alpha$. The pairs $(X_\alpha,\chartfuns)$ are
      called {\bf differentiable charts} and the $\frac{\partial
        f}{\partial x_\alpha^j}$ are called the {\bf partial
        derivatives of $f$ with respect to the chart functions};
      \item the integer $N_\alpha$ is uniformly bounded. The lowest
        upper bound is called the {\bf dimension of the differentiable structure}.
      \end{itemize}
\end{defn}
\begin{defn}[Local independence of Lipschitz functions]\label{loc_ind_lip}
  Let $f_1,\cdots,f_n$ be Lipschitz functions. We say that they are
  {\bf independent at $x$} if
  \begin{equation}
    \biglip (\lambda_1 f_1 + \cdots \lambda_n f_n)(x)=0
  \end{equation}
  implies
  \begin{equation}
    \lambda_1=\cdots=\lambda_n=0,
  \end{equation}
where $\lambda_i\in\mathbb{R}$.
\end{defn}
Another way of thinking of this notion of independence is the
following. We can define a map
\begin{align}
  \Phi_x: \mathbb{R}^n&\to \mathbb{R}\\
  (\lambda_i)&\mapsto\biglip\left(\sum_{i=1}^n\lambda_if_i\right)(x);
\end{align}
from the properties of $\biglip$ we know that $\Phi_x$ is a
seminorm. The linear independence condition is equivalent to $\Phi_x$
being a norm.
To establish the existence of a measurable differentiable structure the
following principle is usually employed:
\begin{lem}\label{finite-dimensionality}
  Suppose that there is a constant $N$ such that
  if $\left\{f_1,\cdots,f_n\right\}\subset\lipalg X.$ 
  are Lipschitz functions which are 
  independent on a set of positive measure $A$, then $n\le N$.
  Then $X$ admits a measurable differentiable structure whose dimension
  is at most $N$.
\end{lem}
 The proof is a modification of the ideas used to prove 
Theorem \ref{freemodules}. The key point is to show that the partial
derivatives $\chartder[f]$ are measurable, and this can be done by
using Lemma \ref{meas:dep}. Details can be found in
\cite{kleiner_mackay} and \cite[Section 7.2]{keith04}.
\begin{thm}
If the metric measure space $(X,\metric,\meas)$ has a measurable differentiable
  structure, then:
  \begin{itemize}
  \item there exists a measurable cotangent bundle $\cotbund$;
  \item on each chart $\chart$ we have a basis
${\{dx_\alpha^j\}_{j=1}^{N_\alpha}}$ for the fibres of
    $\cotbund[X_\alpha]$;
\item we can define a measurable fibrewise norm by setting:
      \begin{equation}\label{meas:norm}
        \vnorm (v_1,\cdots,v_{N_\alpha})! (x) = \biglip\left\{
            \sum_{j=1}^{N_\alpha}v_j \chartfun
         \right\}(x);
      \end{equation}
  \item if on each chart $\chart$ we define
    \begin{equation}
      d:\lipfun X_\alpha.\to\seccot X_\alpha.
    \end{equation}
    by
    \begin{equation}
      df=\sum_{j=1}^{N_\alpha}\frac{\partial f}{\partial x_\alpha^j}
      d\chartfun,
    \end{equation} 
    then $\vnorm df!\in\elleinfty X_\alpha.$ and
      \begin{equation}
        \vnorm df! (x) = \biglip f (x).
      \end{equation}
    \end{itemize}
Therefore the set of sections $\sections$ is equipped with a norm and
we can define $\ellep \sections.$.
\end{thm}
The following result is similar to \cite[Theorem 4.48]{cheeger99} and
we omit the proof. However, in this setting the Sobolev spaces used by
Cheeger might trivially reduce to the corresponding ${\rm L}^p$
spaces. The Sobolev spaces we work with are therefore different from those
employed by Cheeger; as far as we understand, the crucial point is
that $\lipsob X.$ does not need to inject in $\ellep X.$, so in this
setting there is no analogue of the uniqueness statement in
\cite[Theorem 4.47]{cheeger99}.
\begin{thm}\label{lip:sobolev}
If the metric measure space $(X,\metric,\meas)$ has a measurable differentiable
  structure, define
\begin{equation}
  \presob X.=\left\{f\in\lipfun X.\cap\ellep X.: df \in \ellep \sections.\right\}
\end{equation} 
and
  \begin{equation}
    {\vnorm f!}_{\lipsob X.} = {\vnorm f!}_{\ellep X.} + {\left\|df 
\right\|}_{\ellep \sections.};
  \end{equation}
then $\left(\presob X.,{\vnorm\cdot!}_{\lipsob X.}\right)$ is a normed
vector space whose completion is denoted by  $\lipsob X.$ (Sobolev
space). The space $\lipsob X.$ has the following properties:
\begin{itemize}
\item for $p>1$ the norm is bi-Lipschitz equivalent to a uniformly
  convex norm;
\item for $p>1$ it is reflexive;
\item as $\lipsob X.$ bi-Lipschitz embedds in $\ellep X.\times
\ellep \sections.$, to each $h\in\lipsob X.$ we can uniquely assign a
pair $(g,\gamma)$ with $g\in\ellep X.$ and
$\gamma\in\ellep \sections.$.
In particular there is a $1$-Lipschitz map
\begin{align}
  J:\lipsob X.&\to\ellep X.\\
  (g,\gamma)&\mapsto g;
\end{align}
\item the exterior differential $d$ extends to $\lipsob X.$
giving a map
\begin{align}
d:\lipsob X.&\to\ellep\sections.\\
(g,\gamma)&\mapsto \gamma;
\end{align}
\end{itemize}
\end{thm}
We now note that $\lipsob X.$ is dense in $\ellep X.$. The operator
$d$ is therefore densely defined in $\ellep X.$. We recall the following
definition from Functional Analysis \cite[Chapter 2]{brezis_fun}:
\begin{defn}\label{d:closability}
The exterior differential
\begin{equation}
  d:\lipsob X.\subset\ellep X.\to\ellep\sections.
\end{equation}
is said to be a {\bf closed operator} if
  $f_n\to f$ in $\ellep X.$ and $df_n\to\gamma$ in $\ellep\sections.$
  implies that $f\in\lipsob X.$ and $df=\gamma$.
\end{defn}
The following Proposition will be used in Section \ref{choice_sec}.
\begin{prop}\label{d:closabilityprop}
  The map 
$$
J:\lipsob X.\to\ellep X.
$$ is injective if and only if the operator
$$
d:\lipsob X.\subset\ellep X.\to\ellep\sections.
$$ is closed.
\end{prop}
\begin{proof}
  Assume that $J$ is injective. Suppose that $f_n\to f$ in $\ellep
  X.$ and $df_n\to\gamma$ in $\ellep\sections.$. Then $(f_n,df_n)$ is
  a Cauchy sequence in $\lipsob X.$ and so it converges to a limit
  $(f,\gamma)$. As $J$ is injective, $(f,\gamma)=(f,df)$ showing that
  $d$ is closed. Conversely, assume that $J$
  is not injective; we can find $(f,\gamma)\in\lipsob X.$ with
  $\gamma\ne df$. In particular, there is a sequence of Lipschitz functions
  $f_n$ with $f_n\to f$ in $\ellep X.$ and $df_n\to\gamma$ in $\ellep
  X.$. As $\gamma\ne df$, $d$ is not closed.
\end{proof}
\section{Finite dimensionality and derivations}\label{sec_fin_dime}
 In this section we prove a finite dimensionality result, that
is the existence of a measurable differentiable structure, by
assuming an inequality in which
the local
Lipschitz constant of a function
is controlled by finitely many derivations. We have
decided to name this inequality \eqref{derivation-inequality} the
{ \bf reverse infinitesimal derivation inequality}. This condition
should be compared with the ``Lip-derivation'' inequality(ies) studied in 
\cite{gong11}. One difference is that we allow the constant in the
inequality to vary with the point (so we use $\lambda(x)$) but
independently of the Lipschitz functions. The reverse infinitesimal
inequality should also be compared with the ``Lip-lip'' inequality of
\cite{keith04}. An explanation about the terminology, ``Lip'' denotes
the local Lipschitz constant $\biglip$ and ``lip'' the local Lipschitz
constant $\smllip$. Our argument is based on measure
theory and uses linear algebra to imply finite dimensionality.
The interplay between measure theory and linear algebra
is made possible by an approximation argument, Lemma
\ref{continuity}, whose proof uses the notion of precise
representative which we now recall.
\begin{defn}[Precise representative]
  Let $g\in\elleoneloc X.$. If the Lebesgue differentiation theorem 
  holds (e.g.~if $\mu$ is doubling) we can choose for $g$ the {\bf precise
  representative} defined as follows:
  \begin{equation}
    \precise{g}(x)=
    \begin{cases}
      \lim_{r\searrow0}\avint_{\ball x,r.}g(y)\,d\mu(y)&\text{if the
        limit exists}\\
      0&\text{otherwise.}
    \end{cases}
  \end{equation}
  In this section if $D$ is a derivation we will use the notation
  $\precise{D}f$ for the precise representative of $Df$.
\end{defn}
\begin{prop}\label{lebder}
  Let $A\subset X$, $\mu(A)>0$ and
  $\left\{f_1,\cdots,f_n\right\}\subset\lipalg A.$. There is a
  measurable subset $A'\subset A$ such that $\mu(A\setminus A')=0$ and
  for all $x\in A'$, $\left\{c_1,\cdots,c_n\right\}\subset\real$
  \begin{equation}\label{trentanove}
    \precise{D}\left(\sum_{i=1}^nc_if_i\right)(x)=\sum_{i=1}^nc_i\precise{D}f_i(x). 
  \end{equation}
\end{prop}
\begin{proof}
  Let $A'\subset A$ be a full measure subset of $A$ such that
  for each $x\in A'$:
  \begin{equation}
    \precise{D}f_i(x)=\lim_{r\searrow0}\avint_{\ball x,r.}Df_i(y)\,d\mu(y);
  \end{equation}
  if $\left\{c_1,\cdots,c_n\right\}\subset\real$, then
  \begin{equation}
    \lim_{r\searrow0}\avint_{\ball
      x,r.}D\left(\sum_{i=1}^nc_if_i\right)(y)\,
    d\mu(y)=\lim_{r\searrow0}\avint_{\ball
      x,r.}\sum_{i=1}^nc_iDf_i(y)\, d\mu(y);
  \end{equation}
  therefore the limit 
    \begin{equation}
    \lim_{r\searrow0}\avint_{\ball
      x,r.}D\left(\sum_{i=1}^nc_if_i\right)(y)\,
    d\mu(y)
    \end{equation}
exists and equals 
\begin{equation}
\sum_{i=1}^nc_i\precise{D}f_i(x)
\end{equation}
showing that \eqref{trentanove} holds.
\end{proof}
\begin{thm}\label{der-finite-dimensionality}
  Let $(X,\metric,\meas)$ be a doubling metric measure space.
  Assume that:
  \begin{itemize}
   \item there are $N$ derivations $\dervec$ and a nowhere vanishing
  $\lambda\in\elleinfty X.$;
   \item for any Lipschitz function $f$, there is a set $\Omega_f$
     such that 
  \begin{align}
    \mu(\Omega_f)&=0;\\ \label{derivation-inequality}
    \max_{j=1,\cdots,N}|D_jf(x)|&\ge\lambda(x)\biglip f(x)\quad\forall
    x\in{}^c\Omega_f;
  \end{align}
  \end{itemize}
    then $X$ admits of a measurable differentiable structure whose dimension
  is at most $N$. The relation \eqref{derivation-inequality} will be referred
  to as the {\bf reverse infinitesimal derivation inequality}.
\end{thm}
 As we already said, the 
 proof relies on the following approximation argument.
The point is that if we have a linear dependence relation where the
$c_i$ are functions, we would like to treat them as constants so that
the linear dependence relation ``localizes'' at the points of a full
measure subset. 
\begin{lem}\label{continuity}
  Assume that the derivation inequality \eqref{derivation-inequality}
  holds. Let $A\subset X$, $\mu(A)>0$ and
  $\left\{f_1,\cdots,f_n\right\}\subset\lipalg A.$. There is a
  measurable subset $A'\subset A$ such that $\mu(A\setminus A')=0$ and
  for all $x\in A'$, $\left\{c_1,\cdots,c_n\right\}\subset\real$
  \begin{align}\label{continuityrel1}
        \max_{j=1,\cdots,N}\left|\precise{D_j}\left(\sum_{i=1}^nc_if_i
        \right)(x)\right|&\ge\lambda(x)\biglip\left(\sum_{i=1}^nc_if_i
        \right)(x)\\ \label{continuityrel2}
        \left|\precise{D_j}\left(\sum_{i=1}^nc_if_i
        \right)(x)\right|&\le\vnorm D_j!\biglip\left(\sum_{i=1}^nc_if_i
        \right)(x)\quad\text{for $j=1,\cdots,N$.}
  \end{align}
\end{lem}
\begin{proof}
  Let $\Psi\subset S^{n-1}$ be a countable dense subset of the unit
  sphere and let
  \begin{equation}
    \Psi(f_1,\cdots,f_n)\equiv\left\{\sum_{i=1}^n a_i
    f_i:\left(a_1,\cdots, a_n\right)\in\Psi\right\}.
  \end{equation}
  Given a function $f\in\lipalg A.$ let $\Omega_f$ denote the set
  where either one of the followings fails:
    \begin{align}\label{cinquantadue}
        \max_{j=1,\cdots,N}\left|\precise{D_j}f
        (x)\right|&\ge\lambda(x)\biglip f
        (x)\\\label{cinquantatre}
        \left|\precise{D_j}f
        (x)\right|&\le\vnorm D_j!\biglip f
        (x)\quad\text{for $j=1,\cdots,N$;}
  \end{align}
  by assumption and by Theorem \ref{rever-der-ineq}, $\mu(\Omega_f)=0$.
Let   \begin{equation}
    \Omega_\Psi(f_1,\cdots,f_n)=\bigcup_{f\in\Psi(f_1,\cdots,f_n)}\Omega_f;
  \end{equation}
% and let $A'\subset A\setminus \Psi(f_1,\cdots,f_n)$ be a full measure subset
% where the conclusions of Proposition \ref{lebder} hold for
% $\left\{D_1,\cdots,D_N\right\}$ and $\left\{f_1,\cdots,f_n\right\}$.
then  for  $x\in A\setminus\Omega_\Psi(f_1,\cdots,f_n)$,
\eqref{cinquantadue} and \eqref{cinquantatre} hold for any multiple
$cf$ with
$f\in\Psi(f_1,\cdots,f_n)$ and $c\in\real$. Let us fix some
$\left\{c_1,\cdots,c_n\right\}\subset\real$; for any $\epsi>0$ there
is a 
\begin{equation}
  \left(b_1,\cdots,b_n\right)=b\cdot\left(a_1,\cdots, a_n\right)
\end{equation}
such that
\begin{align}
  \left(a_1,\cdots, a_n\right)&\in\Psi;\\
  \sum_{i=1}^n|c_i-b_i|&\le\epsi;
\end{align}
in particular
  \begin{align}\label{cinquantaotto}
        \max_{j=1,\cdots,N}\left|\precise{D_j}\left(\sum_{i=1}^nb_if_i
        \right)(x)\right|&\ge\lambda(x)\biglip\left(\sum_{i=1}^nb_if_i
        \right)(x)\\ \label{cinquantanove}
        \left|\precise{D_j}\left(\sum_{i=1}^nb_if_i
        \right)(x)\right|&\le\vnorm D_j!\biglip\left(\sum_{i=1}^nb_if_i
        \right)(x)\quad\text{for $j=1,\cdots,N$.}
  \end{align}
Let
\begin{align}
  C_1&=\bigvee\lipnorm f_i!;\\
  C_2&=\bigvee\vnorm D_j!;\\
\end{align}
then
\begin{equation}
  \left|\precise{D_j}f_i(x)\right|\le C_1C_2.
\end{equation}
 We now make two
estimates:\begin{equation}
\begin{split}
  \biglip\left(\sum_{i=1}^n(c_i-b_i)f_i\right)(x)&\le\sum_{i=1}^n|c_i-b_i|\biglip
  f_i(x)\\
&\le\epsi C_1;
\end{split}
\end{equation}
and
\begin{equation}
  \begin{split}
    \left|\precise{D_j}\left(\sum_{i=1}^nb_if_i\right)(x)
    -
    \precise{D_j}\left(\sum_{i=1}^nc_if_i\right)(x)\right|
    &=\left|\sum_{i=1}^n(b_i-c_i)\precise{D_j}f_i(x)\right|\\
    &\le\sum_{i=1}^n|b_i-c_i|\left|\precise{D_j}f_i(x)\right|\\
    &\le\epsi C_1C_2.
  \end{split}
\end{equation}
Substitution of the last two estimates into \eqref{cinquantaotto} and
\eqref{cinquantanove} leads to
  \begin{align}
        \max_{j=1,\cdots,N}\left|\precise{D_j}\left(\sum_{i=1}^nc_if_i
        \right)(x)\right|&\ge\lambda(x)\biglip\left(\sum_{i=1}^nc_if_i
        \right)(x)-\epsi C_1C_2-\epsi C_1\lambda(x);\\
        \left|\precise{D_j}\left(\sum_{i=1}^nc_if_i
        \right)(x)\right|&\le\vnorm D_j!\biglip\left(\sum_{i=1}^nc_if_i
        \right)(x)+2\epsi C_1 C_2\quad\text{for $j=1,\cdots,N$.}
  \end{align}
Letting $\epsi\searrow0$ completes the proof of \eqref{continuityrel1}
and \eqref{continuityrel2}.
\end{proof}
\begin{proof}[Proof of Theorem \ref{der-finite-dimensionality}]
  The proof is reduced to Proposition \ref{finite-dimensionality}. We
  assume that there are $n$ Lipschitz
  functions $\left\{f_1,\cdots,f_n\right\}\subset\lipfun X.$
 which are independent at each point
  $x\in A$, where $\mu(A)>0$. We show that $n\le N$ 
  arguing by contrapositive: we show that if $n\ge N$ then the
  functions $\left\{f_1,\cdots,f_n\right\}$ are dependent on a
  positive measure subset of $A$. Without loss of generality we can assume
  that
  $A$ is bounded and replace each $f_i$ by
  \begin{equation}
    \left(f_i\wedge\sup_{A}|f_i|\right)\vee\left(-\sup_{A}|f_i|\right)
  \end{equation}
  so that 
  \begin{equation}
    \left\{f_1,\cdots,f_n\right\}\subset\lipalg X..
  \end{equation} As
  $\lambda$ is nowhere vanishing, we can suppose that $\lambda\ge
  C>0$ on $A$.
  By Lemma \ref{continuity} there is $A'\subset A$ such that
  $\mu(A\setminus A')=0$ and \eqref{continuityrel1} and
  \eqref{continuityrel2} hold on $A'$ (with $\lambda$ replaced by
  $C$):
  \begin{align}\label{continuityrel3}
        \max_{j=1,\cdots,N}\left|\precise{D_j}\left(\sum_{i=1}^nc_if_i
        \right)(x)\right|&\ge C\biglip\left(\sum_{i=1}^nc_if_i
        \right)(x)\\ 
        \left|\precise{D_j}\left(\sum_{i=1}^nc_if_i
        \right)(x)\right|&\le\vnorm D_j!\biglip\left(\sum_{i=1}^nc_if_i
        \right)(x)\quad\text{for $j=1,\cdots,N$.}
  \end{align}
Let us consider the matrix
  \begin{gather} F = 
    \begin{pmatrix}
      \precise{D_1}f_1 &\cdots &\precise{D_1} f_n\\
      \vdots &\cdots&\vdots\\
      \precise{D_N}f_1&\cdots&\precise{D_N}f_n\\
    \end{pmatrix},
  \end{gather}
  with entries in $\elleinfty A'.$. 
  Since $n>N$,there is a measurable $B\subset A'$ with $\mu(B)>0$ and
  $\rank F(x)=k<n$ for $x\in B$.
   Without loss of generality, we can assume that the first
  $k$ columns of $F$ are linearly independent on $B$
  and the first $k+1$ columns of $F$ 
  are linearly dependent on $B$. By Lemma \ref{meas:dep} there are
  $\lambda_i\in\elleinfty B.$ with 
  \begin{equation}
        \mu\left(\left\{x\in B:\forall i,\lambda_i(x)=0\right\}\right)=0.
        \end{equation}
  and
  \begin{equation}
    \sum_{i=1}^{k+1} \lambda_i(x)\,(\precise{D_j}f_i)(x)=0,
  \end{equation}
  for all $x\in B$ and all $j=1,\cdots,N$. We now choose $x\in B$ and define
  $c_i=\lambda_i(x)$. By Proposition \ref{lebder}
  we have
  \begin{equation}
    \left|\precise{D_j}\left(\sum_{i=1}^{k+1}c_if_i
        \right)(x)\right|=0\quad\text{for $j=1,\cdots,N$}
  \end{equation}
and by \eqref{continuityrel3}
  \begin{equation}
    \biglip (\sum_{i=1}^{k+1} c_i f_i)(x)=0.
  \end{equation}
  So the $\left\{f_1,\cdots,f_n\right\}$ are dependent at a.e.~$x\in B$.
\end{proof}
\section{Choice of the chart functions}\label{choice_sec}
 In this section we present some results connected with the
choice of the chart
functions. The starting point is the representation formula \eqref{eq:representation} for
derivations if the space admits a measurable differentiable
structure. This formula  has an interesting consequence: if
 the partial derivatives are known to be derivations,
they give a basis for the module of derivations of the chart
(Corollary \ref{freemodules2}). This naturally leads to the following question:
when are the partial derivatives 
 derivations? We have found two sufficient conditions but we have been unable to
 find a complete answer. If the answer were negative, then there would
 be two kinds of differentiable structures and those in which the
 partial derivatives are also derivations would exhibit a more regular
 behaviour. We next investigate the choice of the chart
functions generalizing the results of \cite{keith04bis}. The main result is that, knowing that the partial
derivatives are derivations, the chart
functions can be chosen among a generating set for the Lipschitz algebra. This
implies immediately that the chart functions can be chosen among
distance functions. 
\begin{lem}\label{representation}
  Suppose the doubling metric measure space $(X,\metric,\meas)$ has a
  measurable differentiable structure, and let $\chart$ be a chart with
  $\chartfuns\subset\lipalg X_\alpha.$. If
  $D\in\dermod[X_\alpha]$ and $f\in\lipalg X_\alpha.$, then
\begin{equation}
  \label{eq:representation}
  Df=\sum_{j=1}^{N_\alpha}\chartder[f]D\chartfun.
\end{equation}
\end{lem}
 Before giving the proof of the Lemma we will restate part of
Lemma \ref{continuity} and of Proposition \ref{lebder}. The point is
that in the proof of Lemma \ref{continuity} the proofs of the
statements of \eqref{continuityrel1} and \eqref{continuityrel2} are
indepedent. While \eqref{continuityrel1} depends on the reverse infinitesimal derivation inequality \eqref{derivation-inequality}, \eqref{continuityrel2} is just
a consequence of the localized derivation inequality
\eqref{rever-der-ineq}
(which is true in a doubling metric space or, more generally, in any
metric measure space where the Lebesgue Differentiation Theorem
holds).
\begin{lem}\label{continuity2}
  Let $A\subset X$ be a measurable subset of positive measure, let
  $$\left\{f_1,\cdots,f_n\right\}\subset\lipalg A.$$ and let
  $$\left\{D_1,\cdots,D_N\right\}\subset\dermod[A].$$ There is a
  measurable subset $A'\subset A$ such that $\mu(A\setminus A')=0$ and
  for all $x\in A'$, $\left\{c_1,\cdots,c_n\right\}\subset\real$
  \begin{align}\label{continuityrel4}
        \left|\precise{D_j}\left(\sum_{i=1}^nc_if_i
        \right)(x)\right|&\le\vnorm D_j!\biglip\left(\sum_{i=1}^nc_if_i
        \right)(x)\quad\text{for $j=1,\cdots,N$.}\\
        \label{linearcontinuity}
       \precise{D_j}\left(\sum_{i=1}^nc_if_i\right)(x)&=\sum_{i=1}^nc_i\precise{D_j}f_i(x).\quad\text{for $j=1,\cdots,N$.}\\ 
  \end{align}
\end{lem}
\begin{proof}
  [Proof of Lemma \ref{representation}]
Without loss of generality we assume that $X_\alpha$ is bounded,
 $\meas(X_\alpha)<\infty$.
 We will show that given $f\in\lipalg X_\alpha.$ and
$D\in\dermod[X_\alpha]$, there is a
measurable subset $C_{f,D}\subset X_\alpha$ with
$\meas(X_\alpha\setminus C_{f,D})=0$ and for all $z\in C_{f,D}$,
\begin{equation}\label{eq:61c}
  \precise{D}f(z)=\sum_{j=1}^{N_\alpha}\chartder[f](z)\precise{D}\chartfun(z).
\end{equation}
This will imply \eqref{eq:representation}.
We apply Lemma \ref{continuity2} with $A=X_\alpha$,
$$
  \left\{D_1,\cdots,D_N\right\}=\left\{D\right\},
$$ and $$
\left\{f_1,\cdots,f_n\right\}=\left\{f,x^1_\alpha,\cdots,x^{N_\alpha}_\alpha\right\},
$$ to obtain a measurable subset $A_{f,D}\subset X_\alpha$ such that
\begin{itemize}
\item $\meas(X_\alpha\setminus A_{f,D})=0$
\item for all $z\in A_{f,D}$,
  $\left\{a,c_1,\cdots,c_{N_\alpha}\right\}\subset\real$,
  \begin{align}
    \label{eq:61a}
        \left|\precise{D}\left(af+\sum_{j=1}^{N_\alpha}c_j\chartfun
        \right)(z)\right|&\le\vnorm D!\biglip\left(af+\sum_{j=1}^{N_\alpha}c_j\chartfun
        \right)(z).\\\label{eq:61b} 
        \precise{D}\left(af+\sum_{j=1}^{N_\alpha}c_j\chartfun
        \right)(z)&=a\precise{D}f(z)+\sum_{j=1}^{N_\alpha}c_j\precise{D}\chartfun(z).
  \end{align}
\end{itemize}
From the definition of measurable differentiable structure there are a
measurable subset $B_{f,D}\subset X_\alpha$ and maps:
$$
\chartder[f]:B_{f,D}\to\real\quad\text{for $j=1,\cdots,N_\alpha$}
$$ such that
\begin{itemize}
\item $\meas(X_\alpha\setminus B_{f,D})=0$;
\item $\forall z\in B_{f,D}$
$$
\sup_{j=1,\cdots,N_\alpha}\sup_{z\in
  B_{f,D}}\left|\chartder[f](z)\right|\le C\glip f.,
$$
\item $\forall z\in B_{f,D}$
  \begin{equation}
    \label{eq:61d}
    \biglip\left(f-\sum_{j=1}^{N_\alpha}\chartder[f](z)\chartfun\right)(z)=0.
  \end{equation}
\end{itemize}
If we let $C_{f,D}=A_{f,D}\cap B_{f,D}$, set
$$
\left(a,c_1,\cdots,c_{N_\alpha}\right)=
\left(-1,\frac{\partial f}{\partial x_\alpha^1}(z),\cdots,
\frac{\partial f}{\partial x_\alpha^{N_\alpha}}(z)\right)
$$ in \eqref{eq:61a}, apply \eqref{eq:61d} and finally use
\eqref{eq:61b}
we deduce \eqref{eq:61c}.
\end{proof}
\begin{cor}
  Suppose the doubling metric measure space $(X,\metric,\meas)$ has a measurable differentiable
structure which has dimension $N$. Then $\dermodsp$ has rank locally
bounded by $N$, in particular Theorem \ref{freemodules} applies.
\end{cor}
\begin{proof}
  As the charts measurably partition $X$, it suffices to show that if
  $U\subset X_\alpha$ has positive measure and if the derivations
  $$
  \left\{D_1,\cdots,D_n\right\}\subset\dermod[U]
 $$ are linearly independent, then $n\le N_\alpha\le N$. We argue by
 contrapositive, that is, by showing that if $n>N_\alpha$, the derivations
 cannot be linearly independent. Using Lemma \ref{continuity2} we find
 a set $U'\subset U$ with $\meas(U\setminus U')=0$ and
 \eqref{continuityrel3} and \eqref{continuityrel4} hold for the
 derivations $\left\{D_1,\cdots,D_n\right\}$ and the chart
 functions. We now consider the matrix
$$ 
F=
\begin{pmatrix}
 \precise{D}_1x^1_\alpha&\cdots&\precise{D}_1x^{N_\alpha}_\alpha\\
 \vdots&\cdots&\vdots\\
 \precise{D}_nx^1_\alpha&\cdots&\precise{D}_nx^{N_\alpha}_\alpha\\
\end{pmatrix}
$$ with entries in $\elleinfty U'.$. As $n>N_\alpha$ there is a measurable
subset $U''\subset U'$ of positive measure on which the rank of $F$ is
$k < n$. Without loss of generality we can assume that the first $k$
rows are linearly independent and the first $k+1$ rows are linearly
dependent. By Lemma \ref{meas:dep} there are $k+1$ functions
$\lambda_i\in\elleinfty U''.$ such that 
\begin{align}
  {\vnorm \lambda_i!}_{\elleinfty A.}&\le1\\
      \sum_{i=1}^{k+1}\lambda_i(x)\precise{D}_i\chartfun(z)&=0\quad\text{for
        a.e.~$z\in U''$ and $j=1,\cdots,N_\alpha$} \\
    \mu\left(\left\{x:\forall i,\lambda_i(x)=0\right\}\right)&=0.
\end{align}
From \eqref{eq:representation} we deduce that 
$$
\sum_{i=1}^{k+1}\lambda_i\precise{D}_i=0
$$ in $\dermod[U'']$ showing that the derivations
$\left\{D_1,\cdots,D_{k+1}\right\}$ are not linearly independent.
\end{proof}
\begin{cor}\label{freemodules2}
Suppose the doubling metric measure space $(X,\metric,\meas)$ has a measurable differentiable
structure, and let $\chart$ be a chart. 
  If the partial derivatives
  $\{\chartder\}_{j=1}^{N_\alpha}\subset\dermod[X_\alpha]$, then
  $\dermod[X_\alpha]$ is free and $\{\chartder\}_{j=1}^{N_\alpha}$ is
  a basis. 
\end{cor}
 We now give two criteria for the ``partial derivatives'' to be
derivations. 
\begin{lem}
Suppose the doubling metric measure space $(X,\metric,\meas)$ has a measurable differentiable
structure, and let $\chart$ be a chart. If 
the map
$$
J:\lipsob X.\to\ellep X.
$$ is injective, then the maps
  \begin{gather}
    \chartder : \lipalg X_\alpha. \to \elleinfty X_\alpha.\\
    f\mapsto \chartder[f]\\
  \end{gather}
are derivations.
\end{lem}
\begin{proof}
Recall Definition \ref{derivationsdef} where the axioms that
derivations have to satisfy are stated. The partial derivatives
$\chartder$ satisfy
 linearity, boundedness and the
Leibniz rule by definition. We have to
 check weak* continuity. As usual, there is no loss of
  generality in assuming that $X_\alpha$ is bounded and of finite measure. 
 Let  $\{f_k\}\subset\lipalg
  X_\alpha.$ and $f_k\to f$ weak* in $\lipalg X_\alpha.$, that is,
  $f_k\to f$ uniformly with $\glip f_k.$ uniformly bounded. We have to
  show that
$$
\chartder[f_k]\to\chartder[f]
$$ weak* in $\elleinfty X_\alpha.$. We will prove the following
statement: for any subsequence $\{f_{k_l}\}\subset\{f_k\}$ we can pass
to a further subsequence $\{f_{\tilde{k}_l}\}\subset\{f_{k_l}\}$ such
that
$$
\chartder[f_{\tilde{k}_l}]\to\chartder[f]
$$ weak* in $\elleinfty X_\alpha.$. This means that for any $g\in {\rm{
L}}^1(X_\alpha,\meas)$ we have to show that
\begin{equation}\label{eq:62a}
\int_{X_\alpha}\chartder[f_{\tilde{k}_l}]g\,d\mu\to\int_{X_\alpha}\chartder[f]g\,d\mu.
\end{equation} As the sequence $\chartder[f_k]$ is uniformly bounded
in $\elleinfty X_\alpha.$ and as continuous functions are dense in ${\rm
L}^1(X_\alpha,\meas)$, it will suffice to consider $g$ continuous in
\eqref{eq:62a}. We observe that $\{f_{k_l}\}$ is a bounded sequence in
$\lipsob X_\alpha.$ and by reflexivity we can pass to a subsequence
$\{f_{\tilde{k}_l}\}$ such that $f_{\tilde{k}_l}\to h$ weakly in
$\lipsob X_\alpha.$. As $\lipsob X_\alpha.$ bi-Lipschitz embedds in $\ellep X.\times
\ellep \sections.$, we conclude that $h=(f,\gamma)$ and
$df_{\tilde{k}_l}\to \gamma$ weakly in $\ellep\sections.$. This implies
that for every continuous function $g$,
$$
\int_{X_\alpha}\chartder[f_{\tilde{k}_l}]g\,d\mu\to\int_{X_\alpha}\gamma_j
g\,d\mu,$$
but, as $J$ is injective, $\gamma=df$ showing that \eqref{eq:62a} holds.
\end{proof}
\begin{lem}
  Suppose the doubling metric measure space $(X,\metric,\meas)$
  satisfies the hypotheses of Theorem
  \ref{der-finite-dimensionality} (in particular the reverse infinitesimal derivation inequality), and let $\chart$ be a chart. Then the
  maps   \begin{gather}
    \chartder : \lipalg X_\alpha. \to \elleinfty X_\alpha.\\
    f\mapsto \chartder[f]\\
  \end{gather}
are derivations.
\end{lem}
\begin{proof}
  Let us consider the matrix
$$ 
F=
\begin{pmatrix}
 \precise{D}_1x^1_\alpha&\cdots&\precise{D}_1x^{N_\alpha}_\alpha\\
 \vdots&\cdots&\vdots\\
 \precise{D}_Nx^1_\alpha&\cdots&\precise{D}_Nx^{N_\alpha}_\alpha\\
\end{pmatrix}
$$ with entries in $\elleinfty X_\alpha.$. We first show that this
matrix has a.e.~rank $N_\alpha$. Suppose on the contrary that on some
subset $U\subset X_\alpha$ with $\meas(U)>0$ the rank of $F$ is
$k<N_\alpha$. Without loss of generality we can assume that the first
$k$ columns are independent while the first $k+1$ columns are linearly
dependent. By Lemma \ref{meas:dep} there are $k+1$ functions
$\lambda_i\in\elleinfty U.$ such that 
\begin{align}
  {\vnorm \lambda_i!}_{\elleinfty U.}&\le1\\
      \sum_{i=1}^{k+1}\lambda_i(x)\precise{D}_j x^i_\alpha&=0\quad\text{for
        a.e.~$z\in U$ and $j=1,\cdots,N$} \\
    \mu\left(\left\{x:\forall i,\lambda_i(x)=0\right\}\right)&=0.
\end{align}
We choose a subset $U'\subset U$ with $\meas(U'\setminus U)$ and such
that the conclusions of Proposition \ref{lebder} and Lemma
\ref{continuity} hold for the derivations $\{D_1,\cdots,D_N\}$ and the
chart functions $\{x^1_\alpha,\cdots,x^{k+1}_\alpha\}$. For $z\in U'$
application of \eqref{continuityrel1} for 
$$
(c_1,\cdots,c_{k+1})=(\lambda_1(z),\cdots,\lambda_{k+1}(z))
$$ shows that the chart functions
$\{x^1_\alpha,\cdots,x^{k+1}_\alpha\}$ are dependent at $z$, leading
to a contradiction. Therefore, the rank of $F$ is a.e.~$N_\alpha$. So
given
$U\subset X_\alpha$ of positive measure we can find
$V\subset U$ of positive measure and an $N_\alpha\times N_\alpha$
minor of $F$ whose determinant does
not vanish on $V$. Without loss of generality we will assume that
$$
G=\begin{pmatrix}
 \precise{D}_1x^1_\alpha&\cdots&\precise{D}_1x^{N_\alpha}_\alpha\\
 \vdots&\cdots&\vdots\\
 \precise{D}_{N_\alpha}x^1_\alpha&\cdots&\precise{D}_{N_\alpha}x^{N_\alpha}_\alpha\\
\end{pmatrix}
$$ is non singular on $V$. Using an argument similar to that of
Corollary \ref{orthogonality} we can find $V'\subset V$ with
$\meas(V')>0$ and derivations
$\{D'_1,\cdots,D'_{N_\alpha}\}\subset\dermod[V']$ such that
$$
D'_i\chartfun=\delta_i^j.
$$ This shows that the maps
$$
\chi_{V'}\chartder
$$ are derivations (here we use the Representation Formula \eqref{eq:representation}).  Therefore for any $U\subset
X_\alpha$ of positive measure, there is a subset $V\subset U$ of
positive measure such
that the
$$
\chi_{V}\chartder
$$ are derivations. Using an exhaustion argument similar to that in
the proof of Theorem \ref{freemodules} we find a measurable partition
$$
X_\alpha=\bigsqcup_iV_i\sqcup \Omega
$$ with $\meas(\Omega)=0$ and each
$$
\chi_{V_i}\chartder
$$ is a derivation (in $\dermod[X_\alpha]$). Then the
$$
\chartder=\sum_i\chi_{V_i}\chartder
$$ are derivations on the disjoint union $\bigsqcup_iV_i$.
\end{proof} 
 In the next Theorem we prove that chart functions can be chosen
among a generating set for the Lipschitz algebra.
\begin{thm}\label{choice}  Suppose the doubling metric measure space
  $(X,\metric,\meas)$ admits a measurable differentiable structure and that
  for each chart  $\chart$ the partial derivatives are derivations.
  If 
  $\genset$ is a generating set for the Lipschitz algebra $\lipalg X.$,
   the charts can be chosen so that the chart functions belong to
  $\genset$.
\end{thm}
\begin{proof}
  From Corollary \ref{freemodules2} we know that
  $\left\{\chartder\right\}$ is a basis for $\dermod[X_\alpha]$. Given
  any $U\subset X_\alpha$ with $\meas(U)>0$ we can apply Corollary
  \ref{orthogonality} to find $V\subset U$ with $\meas(V)>0$,
  functions $\left\{g'_1,\cdots,g'_{N_\alpha}\right\}\subset \genset$
  and derivations
  $\left\{D'_1,\cdots,D'_{N_\alpha}\right\}\subset\dermod[V]$ such
  that
  $$
  D'_ig'_j=\delta_{i,j}.
$$ Applying Lemma \ref{continuity2} to
$\left\{D'_1,\cdots,D'_{N_\alpha}\right\}$ and
$\left\{g'_1,\cdots,g'_{N_\alpha}\right\}$ we find $V'\subset V$ with
$\meas(V\setminus V')=0$ and for each $z\in V'$,
\begin{equation}
\begin{split}
\max_{i=1,\cdots,N_\alpha}|c_i|&=\max_{j=1,\cdots,N_\alpha}\left|
  \precise{D'}_j\left(\sum_{i=1}^{N_\alpha}c_ig'_i\right)(z)\right|\\
&\le\max_{j=1,\cdots,N_\alpha}\vnorm D'_j!\biglip\left(\sum_{i=1}^{N_\alpha}c_ig'_i\right)(z)
\end{split}
\end{equation}
which shows that the functions
$\left\{g'_1,\cdots,g'_{N_\alpha}\right\}$ are a.e.~independent on
$V$, so $(V,\left\{g'_{i}\right\}_{i=1}^{N_\alpha})$ is a chart. Using and exhaustion argument similar to that in
the proof of Theorem \ref{freemodules} we can ``cover'' (up to a
subset of measure $0$) $X_\alpha$ by
measurable charts such that the chart functions are among the $\genset$.
\end{proof}
\begin{cor}
Under the hypotheses of Theorem \ref{choice} the chart functions can
be chosen among distance functions from points.  
\end{cor}
\begin{proof}
As a consequence of the Stone-Weierstra\ss\ (Theorem \ref{stone-weierstrass}) Theorem,
the distance functions from points are a generating set for $\lipalg X.$.
\end{proof}
\section{Appendix}
\begin{defn}[pointed Lipschitz Algebra]
  Let $(X,\metric,x_0)$ be a pointed metric space, i.e.~a metric space
  with a basepoint $x_0$. We denote the collection of real-valued Lipschitz
  functions on $(X,\metric)$ which vanish at $x_0$ by $\plipalg
  X.$. The set $\plipalg X.$ is a real algebra where multiplication is
  defined as follows: if $f,g\in\plipalg X.$,
  \begin{equation}
    (fg)(x)=(f(x))(g(x)).
  \end{equation}
  For $f\in\plipalg X.$ we define the norm
  \begin{equation}
    \plipnorm f!=\glip f. .
  \end{equation}
This gives $(\plipalg X.,\plipnorm\cdot!)$ the structure of a Banach
algebra \cite[sec.~4.1]{weaver_book99}.
\end{defn}
A reference for the following Theorem is
\cite[sec.~2.2]{weaver_book99}.
\begin{thm}\label{arenseells}
  The Banach space $\plipalg X.$ has a predual, $\arenseels X.$ which
  is separable if $X$ is separable. The space $\arenseels X.$ is
  called the Arens-Eells space and does not depend on the choice of
  the basepoint $x_0$.
\end{thm}
We can now establish the following:
\begin{thm}\label{arenseellsbounded}
  The Banach space $\lipalg X.$ has a predual which is separable if
  $X$ is separable.
\end{thm}
\begin{proof}
  Let $x_0\in X$ and define
  \begin{align}
    \Phi:\lipalg X.&\to\plipalg X.\oplus_\infty l^\infty(X)\\
    f&\mapsto(f-f(x_0),f);
  \end{align}
where we now explain the notation. If $X_1,X_2$ are normed vector
spaces, we denote by $X_1\oplus_\infty X_2$ the normed vector spaces
$X_1\times X_2$ with the norm
\begin{equation}
  \vnorm (x_1,x_2)!_{X_1\oplus_\infty X_2}=\vnorm x_1!_{X_1}\vee\vnorm x_2!_{X_2}.
\end{equation}
This construction is easily generalized for $p\in[1,\infty)$: if $X_1,X_2$
are normed vector space, we denote by $X_1\oplus_p X_2$ the
normed vector space $X_1\times X_2$ with the norm
\begin{equation}
  \vnorm (x_1,x_2)!_{X_1\oplus_p X_2}=\left[\vnorm x_1!_{X_1}^p+\vnorm x_2!_{X_2}^p\right]^{1/p}.
\end{equation} We will denote the dual of a Banach space $Y$ by $\dual Y.$.
It is a well-known fact that
\begin{equation}
\dual  \left(X_1\oplus_p X_2\right). = \dual X_1.\oplus_q \dual X_2.,
\end{equation}
where $q$ is the conjugate exponent of $p$. The set $l^\infty(X)$
denotes the collection of bounded functions of $X$. It is a Banach
space with the sup-norm. The set $l^1(X)$ denotes the collection of
functions $f$ on $X$ for which
\begin{equation}
  \vnorm f!_{l^1(X)}=\sum_{x\in X}|f(x)|<\infty;
\end{equation}
$(l^1(X),\vnorm \cdot!_{l^1(X)})$ is a Banach space and $\dual
l^1(X).=l^\infty(X)$. Furthermore, $l^1(X)$ is separable if $X$ is
separable. We know from Theorem \ref{arenseells} that
\begin{equation}
\dual{\arenseels X.}. = \plipalg X.,
\end{equation}
therefore
\begin{equation}
  \plipalg X.\oplus_\infty l^\infty(X) = \dual\left({\arenseels X.\oplus_1 l^1(X)}\right)..
\end{equation}
From the definition of $\lipnorm\cdot!$ it follows that $\Phi$ is an
isometric embedding. It is a well-know Banach space fact that if a
Banach space $Z$ isometrically embedds in $\dual Y.$, then
$Z$ is isometric to
\begin{equation}
  \dual(Y/Z^\perp).,
\end{equation}
where $Z^\perp$ is the preannihilator of $Z$ in $Y$:
\begin{equation}
  Z^\perp=\left\{y\in Y:\forall z\in Z, z(y)=0\right\}.
\end{equation}
Therefore, there is a quotient of $\arenseels X.\oplus_1 l^1(X)$ which
is a predual of $\lipalg X.$. Moreover, if $X$ is separable this
predual is separable.
\end{proof}
Note that for ``interesting spaces'' $\lipalg X.$ is usually {\bf
  neither separable nor reflexive}. The typical example is $\lipalg
[0,1].$. The pointed Lipschitz algebra ${{\rm Lip}_0([0,1],0)}$
isometrically embedds in $\lipalg [0,1].$. The derivative map
\begin{align}
  D: {{\rm Lip}_0([0,1],0)} &\to {{\rm L}^\infty([0,1],{\rm Leb})} \\
  f&\mapsto f'\\
\end{align} 
is an isometry with inverse the indefinite integral:
\begin{align}
  I: {{\rm L}^\infty([0,1],{\rm Leb})} &\to {{\rm Lip}_0([0,1],0)} \\
  f'&\mapsto \int_0^x f'.\\
\end{align}
A Banach space which is a dual space has, in general, not a unique
predual, but this is true for $\lipalg X.$. This follows from the fact
that $\lipalg X.$ is a Banach algebra. Therefore, we can say that the
Arens-Eells space $\arenseels X.$ is {\bf the} predual of $\lipalg X.$
and on $\lipalg X.$ we can consider {\bf the} weak* topology. If a
Banach space $Y$ is separable, the weak* topology on $\dual Y.$ is
metrizable on each ball \cite[Chapter 3]{brezis_fun}. It is therefore
useful to know when $f_n\to f$ weak* in $\lipalg X.$. A reference for
the following Proposition is \cite[sec.~2.2]{weaver_book99}:
\begin{prop}\label{weak*topology}
  A sequence $\{f_n\}\subset\plipalg X.$ converges to $f\in\plipalg X.$ with
  respect to the weak* topology if and only if it converges to $f$ pointwise.
\end{prop}
Note that when bounded subsets of $X$ are relatively compact, e.g.~in
the case of a doubling metric measure space, then we have that $f_n\to
f$ uniformly on bounded subsets because $\sup_n\glip f_n. <\infty$.
From Proposition \ref{weak*topology} we get
\begin{prop}
  A sequence $\{f_n\}\subset\lipalg X.$ converges to $f\in\lipalg X.$ with
  respect to the weak* topology if and only if it converges to $f$ pointwise.
\end{prop}
\begin{proof}
  This follows from the proof of Theorem \ref{arenseellsbounded}. As
  $\Phi$ is an isometric embedding of $\lipalg X.$ into $\plipalg
  X.\oplus_\infty l^\infty(X)$, the weak* topology on $\lipalg X.$ is
  the relative topology induced by the weak* topology on $\plipalg
  X.\oplus_\infty l^\infty(X)$. Therefore, the result follows from
  Proposition \ref{weak*topology}.
\end{proof}
If one compares the definition of derivations \ref{derivationsdef}
with the original definition introduced by Weaver in \cite{weaver00},
the third condition, i.e.~that derivations preserve weak* converge for
sequences, is replaced by the weak* continuity of derivations. This
means that
\begin{equation}
  D:\lipalg X.\to\elleinfty X.
\end{equation}
is continuous when we consider on both $\lipalg X.$ and $\elleinfty
X.$ the weak* topology (we take ${\rm
L}^1(X,\meas)$ as the predual of $\elleinfty X.$). In the cases we are
studying $X$ is separable and the two conditions agree. This follows
from the following Proposition:
\begin{prop}
  Let $Y$ be a separable Banach space and let
  \begin{equation}
    T: \dual Y.\to \dual Z.
  \end{equation} be a linear map. Then $T$ is continuous with respect
  to the weak* topologies on $\dual Y.$ and $\dual Z.$ if and only if
  $y_n^*\to y^*$ weak* in $\dual Y.$ implies $Ty_n^*\to Ty^*$ weak* in
  $\dual Z.$.
\end{prop}
\begin{proof}
  It is enough to prove that the condition is sufficient. We have to
  show that for any $z\in Z$ the linear map
  \begin{align}
    \phi_z:\dual Y.&\to\real\\
    y^*&\mapsto (Ty^*)(z)
  \end{align}
  is continuous with respect to the weak* topology on $\dual Y.$. It
  suffices to show that $\ker\phi_z$ is weak* closed. Now,
  $\ker\phi_z$ is a linear subspace of $\dual Y.$, so it is convex. By
  the Krein-\v{S}mulian Theorem \cite[Chapter 3]{brezis_fun}, it
  suffices to show that for each $R>0$, $\clball 0, R.\cap\ker \phi_z$
  is closed in the relative weak* topology on the closed ball $\clball
  0, R.$. But as $Y$ is separable, the weak* topology on each $\clball
  0,R.\subset \dual Y.$ is metrizable \cite[Chapter
  3]{brezis_fun} and therefore it is enough to know that $\phi_z$ is
  sequentially continuous.
  \end{proof}
  We now present a proof of Proposition \ref{locality} following the
  argument of \cite[Lemma 13.4]{heinonen07}.
  \begin{proof}[Proof of Proposition \ref{locality}]
    By linearity, it suffices to show that $Df=0$ a.e.~on the set
    \begin{equation}
      B = \{x: f(x)=0\},
    \end{equation}
and if $B$ has null
measure the claim is trivial.
    Let us define
    \begin{equation}
      h_n=
      \begin{cases}
        \sgn f \cdot \sqrt{|f|}& \text{for $|f|\ge 1/n$} \\
        \sqrt{n} f & \text{for $|f|\le 1/n$} \\
        \end{cases};
    \end{equation}
    note that $\{h_n\}\subset\lipalg X.$ but it might not be a bounded
    sequence. Note also that $h_n$ vanishes on $B$. Now, the sequence
    $\{|h_n|\cdot h_n\}\subset\lipalg X.$ is uniformly bounded and
    converges pointwise to $f$. In fact,
    \begin{equation}
      |h_n|\cdot h_n=
      \begin{cases}
        f & \text{for $|f|\ge 1/n$} \\
        n f|f| & \text{for $|f|\le 1/n$} \\
        \end{cases}.
    \end{equation}
Let $C\subset B$ be a subset of finite positive measure. Then, as $|h_n|\cdot h_n\to f$ weak* in
$\lipalg X.$, then $D(|h_n|\cdot h_n)\to Df$ weak* in ${\rm
L}^1(X,\meas)$. In particular,
\begin{equation}
\begin{split}
  \int_X Df\cdot \chi_C\,d\mu&=\lim_{n\to\infty}\int_XD(|h_n|\cdot
  h_n)\cdot \chi_C\,d\mu\\
&=\lim_{n\to\infty}\int_X\left(D|h_n|\cdot h_n+Dh_n\cdot
  |h_n|\right)\cdot \chi_C\,d\mu\\
&=0.
\end{split}
\end{equation}
  \end{proof}
In the previous sections, we have referred to Stone-Weierstra\ss\
Theorem. The classical Stone-Weierstra\ss\ Theorem pertains to the
Banach Algebra $C(X)$ of continuous functions on a compact space $X$,
where the norm is the sup-norm. There is an analogue of this result in
the setting of Lipschitz algebras. From now to the end of this section
{\bf we will assume that} $X$ {\bf is bounded}. 
The condition of separating points
is replaced by an uniform condition:
\begin{defn}A subalgebra $\subalg\subset\lipalg X.$ is said to
  separate points uniformly if there is a constant $M>0$ such that for
  any pair of points $x_1,x_2\in X$ there is an $f\in\subalg$ such
  that 
  \begin{equation}
    |f(x_1)-f(x_2)|=\metric(x_1,x_2) \qquad \text{and} \qquad
    \lipnorm f!\le M.
  \end{equation}
\end{defn}
We state the result for $\lipalg
X.$, in \cite[Chapter~4]{weaver_book99} one can find the proof for
$\plipalg X.$.
\begin{thm}\label{stone-weierstrass}[Stone-Weierstra\ss\ for Lipschitz algebras]
  If $\subalg\subset\lipalg X.$ is a weak* closed subalgebra (with the
  same unity as $\lipalg X.$) that
  separates points uniformly, then $\subalg=\lipalg X.$.
\end{thm}
\bibliographystyle{alpha}
\bibliography{lip_alg_biblio,analysis_metric,graduate_books}
				% note that in
                                % concatenating files for the
                                % bibliography, white spaces are AVOIDED
\end{document}